\documentclass{amsart}

\usepackage[a4paper,margin=3.cm]{geometry}
\usepackage[dvipsnames]{xcolor}

\usepackage{amsopn}

\newtheorem{assumption}{Assumption}

\usepackage{comment}
\usepackage{hyperref}

\newcommand{\RE}{\mathbb{R}}
\newcommand{\R}{\mathbb{R}}
\newcommand{\XX}{\mathbb{X}}
\newcommand{\YY}{\mathbb{Y}}

\newcommand{\bx}{\mathbf{x}}
\newcommand{\by}{\mathbf{y}}
\newcommand{\bs}{\mathbf{s}}
\newcommand{\CS}{\mathcal{S}}
\newcommand{\CF}{\mathcal{F}}
\newcommand{\CG}{\mathcal{G}}

\newcommand{\bZ}{\mathbf{Z}}
\newcommand{\bz}{\mathbf{z}}

\newcommand{\bY}{\mathbf{Y}}
\newcommand{\PP}{\mathbb{P}}
\newcommand{\CN}{\mathcal{N}}

\newcommand{\CI}{\mathcal{I}}
\newcommand{\CM}{\mathcal{M}}
\newcommand{\CB}{\mathcal{B}}

\newcommand{\CD}{\mathcal{D}}
\newcommand{\CW}{\mathcal{W}}
\newcommand{\E}{\mathbf{E}}

\DeclareMathOperator{\tr}{tr}

\DeclareMathOperator{\Cov}{Cov}
\newcommand{\psd}[1]{\mathcal{S}^{+}_{#1}}
\newcommand{\iid}{\stackrel{\text{iid}}{\sim}}
\newcommand{\id}{\mathbf{I}}
\newcommand{\convp}{\xrightarrow{\PP}}
\newcommand{\convd}{\xrightarrow{\CL}}

\newcommand{\eps}{\epsilon}
\newcommand{\CL}{\mathcal{L}}
\newcommand{\CK}{\mathcal{K}}

\newcommand{\sym}[1]{\mathcal{S}_{#1}}
\def\lnorm{\lVert}
\def\rnorm{\rVert}
\newcommand{\veec}{{\bf vec}}

\newtheorem{definition}{Definition}[section]
\newtheorem{lemma}[definition]{Lemma}
\newtheorem{proposition}[definition]{Proposition}

\newtheorem{remark}{Remark}

\newtheorem{example}{Example}
\newtheorem{theorem}[definition]{Theorem}

\title[Convolutional NN]{Large deviation principles for convolutional \\ Bayesian neural networks } 
\author{ Federico Bassetti$^*$ ,  Vassili De Palma$^\dag$
 ,Lucia Ladelli$^\ddag$
}
\thanks{$^*$ email: federico.bassetti@polimi.it; address:  Dipartimento di Matematica,  Politecnico di Milano,  Milano, Italy.}
\thanks{$^\dag $  email:  vdepalma@ethz.ch; address: Departement Informatik, ETH Z\"urich, Switzerland}
\thanks{$^\ddag$  email: lucia.ladelli@polimi.it;  address: Dipartimento di Matematica,  Politecnico di Milano,  Milano, Italy.}

\begin{document}

\maketitle

\begin{abstract}
While suitably scaled CNNs with Gaussian initialization are known to converge to Gaussian processes as the number of channels diverges, little is known beyond this Gaussian limit. We establish a large deviation principle (LDP) for convolutional neural networks in the infinite-channel regime.
We consider a broad class of multidimensional CNN architectures characterized by general receptive fields encoded through a patch-extractor function satisfying mild structural assumptions. Our main result establishes a large deviation principle for the sequence of conditional covariance matrices under a Gaussian prior distribution on the weights. We further derive an LDP for the posterior distribution obtained by conditioning on a finite number of observations. In addition, we provide a streamlined proof of the concentration of the conditional covariances and of the Gaussian equivalence of the network.
To the best of our knowledge, this is the first large deviation principle established for convolutional neural networks.%
\end{abstract}

%

%

\section{Introduction}

Convolutional Neural Networks (CNNs) are a prominent class of deep neural networks specifically designed to process data with an underlying grid-like structure. Typical examples include image data, which can be naturally represented as a two-dimensional lattice of pixels. The defining feature of CNNs is the use of convolutional layers, which extract local features through weight sharing and localized receptive fields. This architectural property has led to remarkable empirical success across a wide range of applications.

Despite their widespread use, the theoretical understanding of CNNs—especially in asymptotic regimes—remains significantly less developed than that of fully connected neural networks (FCNNs). In the fully connected setting, a substantial body of work has established that, under suitable scaling and Gaussian initialization of the weights, wide neural networks exhibit a form of \emph{Gaussian equivalence}. In particular, as the width tends to infinity, the network output converges in distribution to a Gaussian process; see, among others, \cite{g.2018gaussian,Hanin2023,lee2018deep}.

At finite width, the output of a Gaussian FCNN can be described as a mixture of Gaussian random variables with a random covariance structure. Under appropriate scaling, the associated sequence of random covariance matrices converges to a deterministic limit, yielding convergence of the network output to a Gaussian process. Beyond weak convergence, recent works have developed quantitative refinements of this limit, including quantitative central limit theorems, e.g. \cite{CePe25,favaro2023quantitative}, as well as moderate and large deviation principles for both the output and the conditional covariance function; see \cite{ABH2026,MacciTrevisanetal,MaPaTo24,Vog24}. 
At various levels of rigor, more complex asymptotic regimes—where also the number of layers or the number of observations diverges—have been investigated; see, among others, 
\cite{aiudi2023,baglioni2024predictive,BaLaRo,hanin2024,Mufan2023,pacelli2023statistical}.

For CNNs, analogous results are considerably more limited. Deep convolutional networks with many channels and Gaussian initialization have been shown to converge to Gaussian processes under suitable scaling; see, e.g., \cite{Rasmussen2019,novak2019bayesian,Yang2019TensorI,yang2020feature}. However, little is known beyond the Gaussian limit itself.

The aim of the present paper is to develop {\it a large deviation theory for convolutional neural networks in the infinite-channel limit}. We consider a broad class of CNN architectures characterized by general receptive fields, described via a patch-extractor function satisfying mild and easily verifiable assumptions that cover most architectures used in practice.

Our main contributions are as follows:
\begin{enumerate}
    \item A large deviation principle (LDP) for the sequence of covariance functions under gaussian prior distribution on the weights (\textbf{Theorem~\ref{prop:gs_ldp0}}).
    \item A LDP for the posterior distribution induced by conditioning on a finite number of observations (\textbf{Proposition~\ref{Prop:posterior}}).
    \item A LDP for the rescaled network output (\textbf{Proposition~\ref{prop:rescalednet}}).
    \item A law of large numbers for the conditional covariances and a central limit theorem for the network output (\textbf{Theorems~\ref{prop:cnn_lln} and \ref{CLT}}).
\end{enumerate}

To the best of our knowledge, this work provides the first large deviation principle for convolutional neural networks.

Concerning point (4), our results are not completely new, even if 
compared with the central limit theorem established in \cite{Rasmussen2019,novak2019bayesian}, our results go beyond the simplified setting of one-dimensional circular-padding architectures and extend to multidimensional CNNs with general receptive fields. Moreover, compared to the general framework of \cite{Yang2019TensorI,yang2020feature}, our approach yields a streamlined proof of the concentration of conditional covariances and the Gaussian equivalence of the network.

The rest of the paper is organized as follows. 
Section \ref{Sec:def} introduces the setting of the problem, presents a few examples, and defines the conditional covariance function. 
Section \ref{Sec:MainRes} states the main results of the paper, in particular the various large deviation principles. 
Section \ref{Sec:generalstruc} reformulates the problem in a slightly more abstract framework, tailored to provide streamlined proofs of the main results. 
Finally, Section \ref{Sec:proofs} and Appendix contain the proofs.

\section{Convolutional Neural Networks: setting of the problem}\label{Sec:def}
We follow the so-called complex layer terminology of \cite{Goodfellow-et-al-2016}, 
where a Convolutional Neural Network (CNN) is described as a small number of relatively complex layers, each composed of several stages. 
Typically, a layer consists of a nonlinear activation (detector stage), a pooling stage, and a convolution stage, with only the latter involving trainable parameters.

In this setting, both the input data and the pre-activations at layer $\ell$ are indexed by a \emph{channel index} $c=1,\dots,C_\ell$ and a \emph{spatial index} $i\in\Lambda_\ell$. 
Each channel represents features computed from a localized region (the \emph{receptive field}) of the previous layer, obtained by applying a discrete convolution. 
This operation uses a set of shared trainable weights (a \emph{convolution mask}) that is applied across all spatial locations, producing weighted sums over local neighborhoods and ensuring spatially consistent feature extraction.

In order to formally define a CNN we need some more notation. 
Although in D-dimensional convolutional neural networks 
the spatial index $i$ is a D-dimensional index $i=(i_1,\dots,i_D)$ 
taking values in a finite subset $\Lambda_\ell \subset \mathbb{Z}^D$, 
for simplicity, and without real loss of generality, 
we label the indices in $\Lambda_\ell$ as $\{1,\dots,N_\ell\}$.
In what follows,
\begin{itemize}
\item $L$ denotes the number of {\it hidden layers};
\item $C_0,C_1,\dots, C_{L+1}$ denote the number of {\it channels} of each layer;
\item $N_0,N_1,\dots,N_{L+1}$ denote the {\it spatial dimension} of each layer;
\item $\CM_\ell$ is the  set of indices (or multi-indices) labeling  the {\it receptive field} of the neurons  at layer $\ell+1$
and $|\CM_\ell |=M_\ell$ is its dimension (sometimes referred to as \emph{mask} dimension);
\item  $T^{(i,\ell)}: \RE^{N_{\ell}} \to \RE^{M_\ell}$ is the
 function extracting the receptive field of neuron
$i=1,\dots,N_{\ell+1}$ at layer $\ell$ from a spatial vector,  
possibly implementing various forms of pooling and the nonlinear activation. 
For a given point $\bz \in \RE^{N_\ell}$ we use the notation
\[
T^{(i,\ell)}(\bz)
=
\bigl[
T^{(i,\ell)}_m(\bz) : m \in \CM_{\ell}
\bigr].
\]
\end{itemize}

Both stride and pooling can be incorporated in the receptive field
formalism through a suitable definition of a patch extractor function $T^{(i,\ell)}$. Some examples will be given in the next subsection.

For a generic input,
\[
\bx=[x_{c,i}:c=1,\dots,C_0,i=1,\dots,N_0] \in \RE^{C_0 \times N_0},
\]
the pre-activation functions at the first layer
$h_{c,i}^{(1)}(\bx)$ are defined, 
for $c=1,\dots,C_1$,
and $i=1,\dots,N_1$,  by
\begin{equation}
h_{c,i}^{(1)}(\bx)
=
\frac{1}{\sqrt{M_0 C_{0}}}
\sum_{c'=1}^{C_{0}}
\sum_{m \in \CM_0}
W^{(0)}_{m,c',c}
\, T^{(i,0)}_m(x_{c',:}),
\end{equation}
and, for layers $\ell = 1,\dots,L$,
\begin{equation}\label{convolutiona.recur1}
   h_{c,i}^{(\ell+1)}(\bx) = \frac{1}{\sqrt{M_\ell C_\ell}} \sum_{c'=1}^{C_\ell} \sum_{m  \in \CM_\ell} W^{(\ell)}_{m,c',c}   
    T^{(i,\ell)}_m \left(h^{(\ell)}_{c',:}(\bx)\right)\,, 
\end{equation}
where  $c=1,\dots,C_{\ell+1}$, $i=1,\dots,N_{\ell+1}$ and 
$W^{(\ell)}_{m,c,c'}$ are the  trainable weights.
Above
$x_{c,:}$ is the vector 
 $x_{c,:}:=(x_{c,1},\dots, x_{c,N_0})^\top$ 
and similarly 
\[
h^{(\ell)}_{c,:}(\bx):=(h^{(\ell)}_{c,1}(\bx),\dots, 
h^{(\ell)}_{c,N_\ell}(\bx))^\top.
\]

The output of the network for an input $\bx$ is 
\[
h_{c,i}^{(L+1)}(\bx), 
 \quad c=1,\dots,C_{L+1}, i=1,\dots, N_{L+1}.
 \]
A crucial feature of convolutional networks is that the weights
$W^{(\ell)}_{m,c',c}$ do not depend on the spatial index $i$:
the same local rule is applied at every location.

\subsection{Examples}\label{Sec:examples}
In building a CNN one needs to choose an order for the various stages, including the non-linear activation $\sigma: \RE \to \RE$.  
The most common ordering is: (1) activation, (2) pooling and padding, and (3) convolution; see Chapter 9 in \cite{Goodfellow-et-al-2016}. 
With this notation, the {\it receptive field / patch extractor operator} formalism 
provides a fairly general way to describe different convolutional 
mechanisms (including stride, padding, pooling, etc.). 
For the sake of clarity, we show below how two  classical network 
architectures can be described within this framework.

\begin{example}[{Convolutional 1D with stride 1, periodic padding and average pooling with stride 2.}]
Here the spatial dimension is identified with points on a
circle, i.e.\ for layer $\ell$ with the discrete torus
$\mathbb{Z}/N_\ell$.
In order to keep the index notation ranging over
$i=1,\dots,N_\ell$, the integers modulo $N_\ell$ are identified with
$\{1,\dots,N_\ell\}$ rather than $\{0,\dots,N_\ell-1\}$.
We describe only the step $\ell \to \ell+1$, assuming for simplicity
that $N_\ell=2N_{\ell+1}$, $M_\ell=3$ and
$\CM_\ell=\{-1,0,1\}$
We want to describe the composition of a linear pooling step that
reduces the spatial dimension from $N_\ell$ to
$N_{\ell+1}=N_\ell/2$ (averaging two non-overlapping neurons at a
time), followed by a convolutional layer.
Define $R^{(0)}:\RE^{N_\ell}\to\RE^{N_{\ell+1}}$ by
\[
R^{(0)}(z)
=
\left[
\frac{z_1+z_2}{2},
\frac{z_3+z_4}{2},
\dots,
\frac{z_{N_\ell-1}+z_{N_\ell}}{2}
\right],
\]
which implements average pooling with window size $2$ and stride $2$. Note that max-pooling can be also used as $R^{(0)}$ operator.
 After the pooling we perform a
convolution with a mask of size $3$ and stride $1$.
To this end we set
\[
R^{(i,\ell)}(z)
=
\bigl(
R^{(0)}(z)_{i-1},
R^{(0)}(z)_{i},
R^{(0)}(z)_{i+1}
\bigr)
\quad \text{mod } N_{\ell+1}.
\]
Choosing $T^{(i,\ell)}(z)=R^{(i,\ell)}(\sigma(z))$ (where $\sigma$ is as usual applied component-wise), 
one obtains
\[
h_{c,i}^{(\ell+1)}(\bx)
=
\frac{1}{\sqrt{3 C_\ell}}
\sum_{c'=1}^{C_\ell}
\sum_{m\in\{-1,0,1\}} \! \!\!
W^{(\ell)}_{m,c',c}  \frac{1}{2} 
\left (\sigma\!\big(h^{(\ell)}_{c',\,2(i+m)-1}(\bx)\big)+\sigma\!\big(h^{(\ell)}_{c',\,2(i+m)}(\bx)\big) \right ).
\]
where $2(i+m)$ is considered 
${\text{mod } N_{\ell+1}}$. 
\end{example}

\begin{example}[{Convolutional 2D with stride 1 and zero padding.}]
Assume that the spatial index $i$ is a two–dimensional multi-index
$i=(i_1,i_2)$.
At each layer
\[
i=(i_1,i_2)\in
\Lambda_\ell=
\{(i_1,i_2)\in\mathbb{Z}^2:
0\le i_1\le\tilde N_\ell,\,
0\le i_2\le\tilde N_\ell
\}.
\]

Clearly $N_\ell=\tilde N_\ell^2$.
Assuming for simplicity $N_\ell=N_0$ for all $\ell$, consider the mask $\CM_\ell=\CM=\{ (\delta_1,\delta_2): \delta_i=0,1,-1\}$ 
of size $3\times3$, with zero padding and unit stride.

For every $i=(i_1,i_2)$ we define
\begin{equation}\label{R0exzeropadd}
\begin{split}
R^{(i,\ell)}(z)
=
(& z_{i_1,i_2},
z_{i_1+1,i_2},
z_{i_1+1,i_2+1},
z_{i_1+1,i_2-1},
z_{i_1-1,i_2},
\\
& z_{i_1-1,i_2-1},
z_{i_1-1,i_2+1},
z_{i_1,i_2-1},
z_{i_1,i_2+1}),
\end{split}
\end{equation}
where $z=[z_i:i\in\Lambda_\ell]$.
Zero padding means that in the right–hand side above
$z_{a,b}$ is set equal to $0$ whenever $(a,b)\notin\Lambda_\ell$.
Taking $T^{(i,\ell)}=\sigma\circ R^{(i,\ell)}$ one obtains
\[
h_{c,i}^{(\ell+1)}(\bx)
=
\frac{1}{\sqrt{9 C_\ell}}
\sum_{c'=1}^{C_\ell}
\sum_{m\in\CM}
W^{(\ell)}_{m,c',c}
\sigma\!\left(
h^{(\ell)}_{c',\,i+m}(\bx)
\right),
\]
with the convention that
$h^{(\ell)}_{c',\,i+m}(\bx)=0$ if $i+m\notin\Lambda_\ell$.
Here the padding is applied before the activation.
Reversing the order of $\sigma$ and $R^{(i,\ell)}$
corresponds to applying the padding after the activation.
The two constructions coincide whenever $\sigma(0)=0$.
\end{example}

\subsection{Conditional Gaussian  structure}\label{Sec:markov-covariances}

In a Bayesian neural network, 
 a prior (probability distribution)  for the weights 
 \[
 \CW=[ W^{(\ell)}_{m,c',c} :\ell=0,\dots,L, m \in \CM_\ell , c=1,\dots,C_{\ell+1}, c'=1,\dots,C_{\ell} ]
 \] is specified.

A very common prior  is the Gaussian prior.
Here and in the rest of the paper $\CN(\mathbf{m},\mathbf{C})$ denotes the Gaussian distribution 
 with mean $\mathbf{m}$ and covariance matrix $\mathbf{C}$. 
In what follows, all the random elements are supposed to be defined on a common  probability space 
$(\Omega,\CF,\PP)$. 
\begin{assumption}[Gaussian prior]\label{A0}
 The weights $W^{(\ell)}_{m,c',c}$ are  independent and per layer identically distributed random variables with
\begin{equation}
    W^{(\ell)}_{m,c',c} \sim \CN\left( 0, {\lambda_\ell^{-1}}\right).
\end{equation}
\end{assumption}

We shall consider a fixed number $P$ of inputs
\[
\bx_{\mu}=[x_{\mu,c,i}: \, c=1,\dots,C_0; i=1,\dots,N_0 ],  \qquad \mu=1,\dots,P.
\]
Note that the collection of all inputs 
$
[x_{\mu,c,i}: \, c=1,\dots,C_0; i=1,\dots,N_0; \mu=1,\dots,P]$
as well as
all the activations at each layer, i.e. 
\[
{H}^{(\ell)}:=[h_{c,i}^{(\ell)}(\bx_\mu): c=1,\dots,C_{\ell},i=1,\dots,N_{\ell},\mu=1,\dots,P],
\]
 can be considered as a 3-D tensor. 

\begin{remark}[Tensor notation]
Given a  $K$ dimensional tensor   
\[
{A}:=[A_{a_1,a_2,\dots,a_K} : a_1=1,\dots,D_1,\dots,a_K=1,\dots,D_K]
\]
with dimensions $D_1,\dots,D_K$,
 the slice along the component $a_j=\alpha$ is the $K-1$ tensor 
$A_{:,:,\dots,\alpha,: ,:,\dots}$
which fix the $j$-th component to be $\alpha$.  
  This is consistent with the same notation that has already been introduced for matrices. 
 More generally, one can consider a slicing which fixes more components, for example if $K=3$, then $A_{\alpha,\beta,:}$ is the vector  $A_{\alpha,\beta,:}=[A_{\alpha,\beta,a_3}: a_3=1,\dots,D_3]$.
 In order to avoid notational burden, we write $\RE^{D_1 \times \dots \times D_K}$ for the set of 
all  $D_1 \times \dots \times D_K$ tensors. 
\end{remark}

Let us introduce the $\sigma$-field containing the information expressed by  the activation functions up to layer $\ell$, that is
\[
\CF^{\ell}:=\sigma \left ( {H}^{(r)}:r=1,\dots,\ell
\right ),
\]
and define for $i,j=1,\dots,N_0$
    \begin{equation*}
        \CK_{i,j}^{(1,C_0)} (\bx_\mu,\bx_\nu): =\frac{1}{ \lambda_{0}  {C_{0}} M_1}  \sum_{c=1}^{C_0 } \sum_{m  \in \CM_1}  T^{(i,0)}_m(x_{\mu,c,:})
       T^{(j,0)}_m  (x_{\nu,c,:}) 
    \end{equation*}
   and, for $\ell=1,\dots,L$ and $i,j=1,\dots,N_\ell$
    \begin{equation*}
        \CK_{i,j}^{(\ell+1, C_{\ell})}(\bx_\mu,\bx_\nu) := \frac{1}{   \lambda_{\ell} {C_{\ell}}M_{\ell}} \sum_{c=1}^{C_{\ell}} 
        \sum_{m  \in \CM_{\ell}}    T^{(i,\ell)}_m \left(h^{(\ell)}_{c,:}(\bx_\mu)\right)  T^{(j,\ell)}_m \left( h^{(\ell)}_{c,:}(\bx_\nu)\right).
    \end{equation*}

\begin{proposition}\label{Prop:conditional_cov} Under {\rm Assumption \ref{A0}},
    for every $\ell = 0,\ldots,L$,
    the collection of random variables $[h_{c,i}^{(\ell+1)}(\bx_\mu): c=1,\dots,C_{\ell+1},i=1,\dots,N_{\ell+1},\mu=1,\dots,P]$, 
    conditionally on $\CF^{\ell}$, 
    are jointly normal with zero means and covariances 
    \[
    \operatorname{Cov}\bigl(
    h_{c,i}^{(\ell+1)}(\bx_\mu),
    h_{c',j}^{(\ell+1)}(\bx_\nu)
    \mid \mathcal{F}^\ell
    \bigr)
    =
    \delta_{c,c'}\,
    \mathcal{K}_{i,j}^{(\ell+1,C_\ell)}(\bx_\mu,\bx_\nu).
    \]
\end{proposition}

The proof is straightforward and is omitted.
We now re-write the (random) covariance tensors appearing 
in the previous Proposition in a more convenient way. 
Define  a function 
$G^{(\ell)}:\RE^{N_\ell \times P} \to \RE^{N_{\ell+1} \times P \times  N_{\ell+1} \times P}$  by 
\begin{equation}\label{def.Gell}
\begin{split}
\bigl[ G^{(\ell)}(z) \bigr]_{i,\mu;j,\nu}
&=
\frac{1}{\lambda_{\ell} M_{\ell}}
\sum_{m \in \mathcal{M}_{\ell}}
 T^{(i,\ell)}_m\left(z_{:,\mu} \right)
T^{(j,\ell)}_m\left( z_{:,\nu} \right),
\\
&
i,j = 1,\dots,N_{\ell+1},
\quad
\mu,\nu = 1,\dots,P .
\end{split}
\end{equation}
and note that   
\[
\mathbf{K}^{(\ell+1,C_\ell)} :=     \Big [ \CK_{i,j}^{(\ell+1,C_\ell)} (\bx_\mu,\bx_\nu) \Big ]_{i,\mu;j,\nu}
= 
\Big[
\frac{1}{C_{\ell}} \sum_{c=1}^{C_{\ell}}
G^{(\ell)}( H^{(\ell)}_{c,:,:} )
\Big]_{i,\mu;j,\nu}
\]
 where  
$[H^{(\ell)}_{c,i,\mu}]_{c,i,\mu}:=[h^{(\ell)}_{c,i}(\bx_\mu)]_{c,i,\mu}$.
Above, in evaluating $G^{(\ell)}(z)$ in $z=H_{c,:,:}^{(\ell)}$ one has 
$T^{(i,\ell)}_m(z_{:,\mu})= T^{(i,\ell)}_m(H^{(\ell)}_{c,:,\mu})=T^{(i,\ell)}_m(h^{(\ell)}_{c,:}(\bx_\mu))$ and similarly for the  $(j,\nu)$ component.

Flattening the 4-dimensional tensor \( G^{(\ell)} \in \mathbb{R}^{N_{\ell+1} \times P \times N_{\ell+1} \times P} \) into a matrix of shape \( N_{\ell+1}P \times N_{\ell+1}P \), by mapping the index pair \((i, \mu)\) to a single index \(k(i,\mu) = \mu + (i-1)P\), and similarly flattening the 2-dimensional tensor \( z \in \mathbb{R}^{N_\ell \times P} \) into a vector of size \( N_\ell P \), we can interpret \( G^{(\ell)} \) as a function
\[
G^{(\ell)}: \mathbb{R}^{N_\ell P} \to \mathcal{S}^+_{N_{\ell+1}P}
\]
where \( \mathcal{S}^+_D \) denotes the set of symmetric, positive semi-definite matrices of size \( D \times D \). The positivity follows  since 
$\mathbf{u}^\top G^{(\ell)}(z) \mathbf{u} \geq 0$.

\begin{remark}
Above and in all the rest of the paper,  we  identify the tensors $G^{(\ell)}(z)$ and $ \mathbf{K}^{(\ell+1,C_\ell)}$  with their flattened matrix form, while retaining the full tensor notation  in indexing for clarity.
\end{remark}

\section{Main results}\label{Sec:MainRes}

Our analysis of the asymptotic properties of convolutional neural networks relies on the following  assumptions.

\begin{assumption}[Infinite channels limit]\label{A1} 
 $L,N_{0},\dots,N_{L+1}, C_0,C_{L+1}$ and $ P$  are  fixed, while  $C_1,\dots, C_{L}$ increase linearly  as a function of $n$, i.e. 
  $C_\ell=C_\ell(n)$ with 
 \[
  \lim_{n \to +\infty} C_\ell(n)/n=\alpha_\ell \in (0,+\infty).
 \]
\end{assumption}

\begin{assumption}[Exponential growth condition]\label{A2_0} 
 For every $\ell=1,\dots,L$ and $i = 1, \ldots, N_{\ell+1}$,  the patch extractor  functions  $T^{(i,\ell)}$ are continuous.  Moreover, 
 there are constants  $0 \leq r <2$ and $0<A_T<+\infty$, $0<B_T<+\infty$  such that
    \begin{equation*}
        |T^{(i,\ell)}_m(z) | \le A_T e^{B_T \|z\|_2^{r}} \qquad \forall z \in \RE^{N_\ell}
    \end{equation*}
    for every $m,i,\ell$.
\end{assumption}

In order to prove a large deviation principle, we shall need a stronger assumption.

\begin{assumption}[Asymptotic Lipschitz  condition]\label{A2_0BIS} 
 The functions $T^{(i,\ell)}$ are continuous and 
for every $z,z'$ in $\RE^{N_\ell}$
\begin{equation*}
|T_m^{(i,\ell)}(z)-T_m^{(i,\ell)}(z')|\le L_T\|z-z'\|_2 + \rho_T(z) + \rho_T(z'),
\end{equation*}
with $\rho_T$ locally bounded and 
$\rho_T(z) = o(\|z\|_2)$ as $\|z\|_2\to\infty$. 
\end{assumption}

In what follows, to simplify the notation, we set 
\[
D_\ell:=N_\ell P
\qquad \text{and} \qquad 
 \mathbf{K}^{(\ell+1,n)} := \mathbf{K}^{(\ell+1,C_\ell(n))} .
\]

\subsection{Covariance concentration and asymptotic normality}
We shall now show that for convolutional neural networks with Gaussian weights, within the infinite channel regime, the random covariance tensor concentrates to a deterministic limit, exactly as in the fully connected case. Consequently, the network output converges in distribution to a Gaussian process.

Following the convention previously introduced, a 2D random tensor $Z$
is distributed as $\CN(\mathbf{0}, C)$, with $C$ a 4D positive symmetric tensor,
if its components $\{Z_{i,\mu}\}$ are jointly Gaussian random variables
with zero mean and covariance
\(
\Cov(Z_{i,\mu}, Z_{j,\nu}) = C_{i,\mu;j,\nu}.
\)
Analogously, a random tensor $Z = [Z_{i,c,\mu}]$ is said to be a 3D Gaussian random tensor
with zero mean and 6D covariance tensor $C_{i,c,\mu;j,c',\nu}$ if its components
$\{Z_{i,c,\mu}\}$ are jointly Gaussian random variables with zero mean and
\(
\Cov(Z_{i,c,\mu}, Z_{j,c',\nu}) = C_{i,c,\mu;j,c',\nu}.
\)

In the rest of the paper, if $X_n$ and $X$ are random elements, $X_n   \stackrel{\CL}{\to} X$ will denote the convergence in law of  $X_n$ 
to $X$, while $X_n   \stackrel{\PP}{\to} X$
denotes convergence in probability.  

\begin{theorem}[Covariance concentration in CNNs]
\label{prop:cnn_lln}
       Let  {\rm\ref{A0}},  {\rm\ref{A1}} and  {\rm\ref{A2_0}} be in force. 
       Then, as $n\to\infty$, 
    \begin{equation*}
    \left(
     \mathbf{K}^{(2,n)},\ldots, \mathbf{K}^{(L+1,n)}
    \right) 
    \stackrel{\PP}{\to}
    \left(
     \mathbf{K}^{(2)},\ldots, \mathbf{K}^{(L+1)}
    \right)
    \end{equation*}
  where  the deterministic tensors  $\mathbf{K}^{(\ell)}$ are 
    recursively defined     for $\ell = 1,\ldots,L$ by
    \begin{equation}
    \label{eq:cnn_nngp_recursion}
\mathbf{K}^{(\ell+1)}_{i,\mu;j,\nu} =
 \frac{1}{  \lambda_{\ell}  M_{\ell}}
        \sum_{m  \in \CM_{\ell}}   \E\left [  T^{(i,\ell)}_m\left( Z^{(\ell)}_{:,\mu} \right)  T^{(j,\ell)}_m \left( Z^{(\ell)}_{:,\nu} \right) \right]  
    \end{equation}
    $i,j=1,\dots,N_{\ell+1}$
    and $\mu,\nu=1,\dots,P $, 
    with $Z^{(\ell)} \sim\CN (\mathbf{0}, \mathbf{K}^{(\ell)})$ and 
\[
\mathbf{K}^{(1)}_{i,\mu;j,\nu}:=\frac{1}{ \lambda_{0}  {C_{0}} M_1}  \sum_{c=1}^{C_0 } \sum_{m  \in \CM_1}  T^{(i,0)}_m(x_{\mu,c,:})
       T^{(j,0)}_m  (x_{\nu,c,:}).
       \]
\end{theorem}

An important consequence of this result is that in the infinite width regime the neural
network simplifies significantly. 
In fact, Proposition \ref{Prop:conditional_cov}
can be rephrased as follows. Let $C\leq C_{\ell}(n)$ be a fixed constant independent of $n$ and define 
 $\mathbf{H}^{(\ell)}=\veec (\veec (H_{1,:,:}^{(\ell)}),\veec(H_{2,:,:}^{(\ell)}),\dots,\veec(H_{C,:,:}^{(\ell)}))$. Then conditionally on 
$\mathcal{F}^{\ell-1}$, $\mathbf{H}^{(\ell)} \sim \CN(\mathbf{0},\mathbf{I}_{C} \otimes \mathbf{K}^{(\ell,n)})$, which is 
\begin{equation}\label{vector-rep}
    \mathbf{H}^{(\ell)}\stackrel{\CL}{=} 
\sqrt{\mathbf{I}_{C} \otimes \mathbf{K}^{(\ell,n)}} \mathbf{Z}
\end{equation}
where $\sqrt{Q}$ is the usual square root of a symmetric and positive semidefinite matrix $Q$ and $\mathbf{Z} \stackrel{}{\sim} \CN(\mathbf{0},\mathbf{I}_{CD_\ell})$.
The continuous mapping theorem, together with Theorem \ref{prop:cnn_lln}, provides a simplified proof of a slightly more general version of the results in \cite{Rasmussen2019, novak2019bayesian}.

\begin{theorem}[Gaussian limit in CNNs]\label{CLT}
     Let   {\rm \ref{A0}},  {\rm\ref{A1}} and  {\rm\ref{A2_0}} be in force. Then for $\ell = 1,\ldots,L+1$ and any fixed positive integer $C$, as $n\to\infty$ one has
    \begin{equation*}
        \left[ h_{i,c}^{(\ell)}\left(\bx_{\mu}\right): i =1,\ldots,N_\ell; c = 1,\ldots,C;\mu = 1,\ldots,P \right]
        \stackrel{\CL}{\to}
        \mathbf{\tilde Z}=[\tilde Z_{i,c,\mu}] \sim \CN \left(\mathbf{0}, \mathbf{\Sigma}^{(\ell)}\right),
    \end{equation*}
    with $\mathbf{\Sigma}^{(\ell)} = \mathbf{I}_{C} \otimes \mathbf{K}^{(\ell)}=\left[\delta_{c,c'} \mathbf{K}^{(\ell)}_{i,\mu,j,\nu}\!:\! c,c' = \!\! 1,\ldots,C; i,j =1,\ldots,N_\ell;\mu,\nu = 1,\ldots,P\right]$, 
     and $\mathbf{K}^{(2)}$, $\ldots, \mathbf{K}^{(L+1)}$ being the deterministic tensors defined 
    by \eqref{eq:cnn_nngp_recursion}.
\end{theorem}

\subsection{Large deviation  for the covariance tensor} 

Let us recall that  
    a sequence of measures $\{\mu_n\}_n$ on the Borel sets  $\CB(\XX)$ of a metric space $\XX$ satisfies a {\it large deviation principle}  (LDP) with rate function $I$ 
and speed $n$ if:
        \begin{enumerate}
        \item for every open set $O \in\  \CB(\XX)$ the following lower bound holds
        \begin{equation*}
        \label{eq:ldp_lower}
            \liminf_{n \to \infty} \frac{1}{n} \log \mu_n(O)  \ge - \inf_{x \in O} I(x) ;
        \end{equation*}
        \item for every closed set $C \in  \CB(\XX)$ the following upper bound holds
        \begin{equation*}
        \label{eq:ldp_upper}
            \limsup_{n \to \infty} \frac{1}{n} \log \mu_n(C)  \le - \inf_{x \in C} I(x) .
        \end{equation*}
    \end{enumerate}
    The rate function  $I:\XX \to \RE \cup \{+\infty\}$ is assumed to be a lower-semicontinuous function. 
    The rate function is said to be \textit{good} if its sub-levels are compact.  
    A sequence of random elements $(X_n)_{n \geq 1}$ is said to satisfy a LDP if the sequence of the corresponding laws, 
    i.e. $\mu_n(\cdot):=\PP\{ X_n \in \cdot\}$, satisfies a LDP.  
See \cite{dembo98}.

   For every $\ell = 0,\ldots,L$ and every   $Q_1 \in \psd{D_{\ell}}$ and $Q_2 \in \psd{D_{\ell+1}}$
   set 
      \begin{equation}
    \label{eq:gs_rate0}
        I_{\ell}\left(Q_{2} \mid Q_{1}\right) = 
        \sup_{Q_0 \in \sym{D_{\ell+1} }} 
        \left\{
        \tr\left(Q_0^\top Q_2 \right) 
        - 
        \log 
        \int_{\R^{D_\ell  }}e^{\tr(Q_0^\top G^{(\ell)}(z))}
        \CN\left(dz\mid \mathbf{0},Q_1 \right)
        \right\},
    \end{equation}
    where the $G^{(\ell)}$'s have been defined in Section \ref{Sec:markov-covariances} by means of the activation function  $\sigma$ and the patch extractor functions $R^{(i,\ell)}$'s and
    $\sym{D}$ is the set of symmetric
    $D \times D$ matrices. 

    \begin{theorem}[LDP in CNNs] 
\label{prop:gs_ldp0}
       Let  {\rm \ref{A0}}, {\rm\ref{A1}} and {\rm\ref{A2_0BIS}} be in force. 
    Then,  the sequence  \[
    \{ (\mathbf{K}^{(2,C_\ell(n))}, \ldots, \mathbf{K}^{(L+1,C_\ell(n))})\}_{n \geq 1}\]   
    satisfies a LDP on $\psd{D_2}\times \dots \times \psd{D_{L+1}}$ with good rate function given by
    \begin{equation*}
        I_{2,\dots,L+1}\left(Q_2, \ldots, Q_{L+1}\right) 
        := 
        \alpha_1 I_{1}\left(Q_{2} \mid  \mathbf{K}^{(1)}\right)
        +
        \sum\limits_{\ell=2}^L \alpha_\ell
        I_{\ell}\left(Q_{\ell+1} \mid Q_{\ell}\right).
    \end{equation*}
\end{theorem}

\begin{remark}[Lipschitz activation]
When $\sigma$ is Lipschitz, the examples in Section~\ref{Sec:examples} satisfy the assumptions of Theorem~\ref{prop:gs_ldp0}, since compositions of Lipschitz functions are Lipschitz and all operations performed by the patch extractor $T^{(i,\ell)}$ are Lipschitz continuous.
\end{remark}

\subsection{LDP under the posterior distribution}\label{Sec:post_A}
In the training of a neural network, a data set is given 
$\{(\mathbf x_\mu ,\mathbf y_\mu)\}_{\mu=1}^P$, where each input $\mathbf x_\mu \in \mathbb R^{C_0 \times N_0}$ is associated with the corresponding label (or response) $\mathbf y_\mu \in \mathbb R^{C_{L+1} \times N_{L+1}}$.

Within a Bayesian framework, to perform any inferential task, a prior distribution is specified on the parameters $\CW$, denoted by $P_{n,\mathrm{prior}}(d\CW)$, together with a likelihood function describing the observation process. The posterior distribution of the parameters $\CW$ given the observations can then be derived and used for inference.

The \emph{likelihood of the labels given the inputs and the outputs} 
\[
(\mathbf y_{1}, \dots, \mathbf y_P) \longmapsto 
L(\mathbf y_{1}, \dots, \mathbf y_P \mid \mathbf s_{1}, \dots, \mathbf s_{P})
\]
represents the conditional density of the random responses 
$[\mathbf Y_1, \dots, \mathbf Y_P]$ evaluated at the point 
$[\mathbf y_1, \dots, \mathbf y_P]$, given that the network outputs satisfy
$h^{(L+1)}(\mathbf x_1) = \mathbf s_{1}$, $\dots, h^{(L+1)}(\mathbf x_P) = \mathbf s_{P}$.
Both $\mathbf y$ and $\mathbf s$ lie in $\mathbb R^{C_{L+1} \times N_{L+1}}$.
The  posterior distribution 
 of  
$\CW$ is by Bayes theorem 
\begin{equation}\label{posterior_infty} 
P_{n,\mathrm{post}}(d\CW |\by_1,\dots,\by_P) 
:=\frac{ L(\by_1,\dots,\by_P| \bs_{1},\dots,\bs_{P} ) P_{n,\mathrm{prior}}(d\CW) 
}{ \int L(\by_1,\dots,\by_P| \bs_{1},\dots,\bs_{P}) P_{n,\mathrm{prior}}(d\CW)  } 
\end{equation}
where 
$\bs_\mu=h^{(L+1)}(\bx_\mu)=h^{(L+1)}(\bx_\mu|\CW)\in \RE^{D_{L+1}}$ with $\mu=1,\dots,P$.

The prior and posterior distributions on the parameters $\CW$ induce corresponding prior and posterior distributions on any function of $\CW$ of interest. In particular, one can derive a posterior distribution for the network outputs $h^{(L+1)}_{c,i}(\mathbf x)$, as well as for the random covariance tensors $\{\mathbf{K}^{(2,n)}, \dots, \mathbf{K}^{(L+1,n)}\}$.

In analogy to a network trained with a quadratic loss function, we shall consider the Gaussian likelihood
\begin{equation}\label{gaussian_likelihood}
L(\by_{1},\dots, \by_P|\bs_{1},\dots,\bs_P)  =\Big(\frac{\beta}{2\pi}\Big)^{{\frac{\bar D}2}} e^{-\frac{\beta}2\sum_{\mu=1}^P \|\bs_{\mu}-\by_{\mu}\|^2}, 
\end{equation}
with $\beta>0$ and 
 $\bar D:=C_{L+1}D_{L+1}=C_{L+1}N_{L+1}P$. This corresponds to assuming the Gaussian error model 
\begin{equation}\label{stat.model.incomponents}
\begin{split}
& Y_{c,i,\mu}=h^{(L+1)}_{c,i} (\bx_\mu) +\varepsilon_{c,i,\mu}\\
& \varepsilon_{c,i,\mu} \iid \CN(0,\beta^{-1}), \quad  c=1,\dots,C_{L+1},\,\, i=1,\dots,N_{L+1}, \,\, \mu=1,\dots,P.\\
\end{split}
\end{equation}

Let $\mathbf{Y}=\veec (\veec (Y_{1,:,:}),\veec(Y_{2,:,:}),\dots)$ be the observations stacked in a vector and
 denote the law of $\mathbf{K}^{(L+1,n)}$ by  $\mathcal{Q}_n$. 
Recall that $\mathbf{K}^{(L+1,n)}$ 
is organized as a $D_{L+1} \times D_{L+1}$ matrix.  

\begin{lemma}
    Assume the prior in {\rm\ref{A0}} and the Bayesian statistical model \eqref{stat.model.incomponents}.
    The conditional   distribution of $\mathbf{K}^{(L+1,n)}$ given $\bY=\by \in \RE^{\bar D}$ (i.e. its posterior distribution) is 
    \[
    \mathcal{Q}_n(dK|\by)= \frac{e^{-\frac12 \Psi ( K|\by) }\mathcal{Q}_n(dK)  }{\int_{\CS^+_{D_{L+1}}}e^{-\frac12 \Psi ( K| \by) } \mathcal{Q}_n(dK) } 
    \]
    where 
$\Psi ( K|\by)=
\beta \by^\top ( \mathbf{I}_{\bar D }+\beta  \mathbf{I}_{C_{L+1}}\otimes K )^{-1}\by
 +\log(\det(\mathbf{I}_{\bar D} +\beta  \mathbf{I}_{C_{L+1}}\otimes K ))$.
\end{lemma}

The proof is based on standard Gaussian manipulations and it is omitted. 
For the full proof of an  analogous result in the FCNN case, see \cite{ABH2026}. 

Combining  Theorem \ref{prop:gs_ldp0}
with the \textit{contraction principle} (see, e.g., Theorem 4.2.1 in \cite{dembo98}), it follows that 
the sequence $\mathcal{Q}_n$ (i.e. the sequence of laws of $\mathbf{K}^{(L+1,n)}$), satisfies 
a LDP with rate function 
\[
I_{L+1}(Q):= \inf_{Q_2,\dots,Q_{L}}
 I_{2,\dots,L+1}\left(Q_2, \ldots,Q_L, Q\right) . 
\]
In point of fact the same LDP is satisfied by the posterior $\mathcal{Q}_n(dK|\by)$ as stated in the next proposition. 

 \begin{proposition}[Posterior LDP in CNNs]\label{Prop:posterior}
 Under the same assumptions of 
{\rm Theorem \ref{prop:gs_ldp0}}, both  $\mathcal{Q}_n(dK)$ 
 and  $\mathcal{Q}_n(dK|\by)$  satisfy a LDP with rate function $I_{L+1}$ defined above. 
\end{proposition}

The proof of the previous proposition is omitted, since based on Theorem \ref{prop:gs_ldp0},
the proof  can be done following the same line as the proof of Proposition 4.2 in \cite{ABH2026}.

In line with the already known results for FCNN proved in  \cite{ABH2026}, the fact that the LDP under the posterior remains the same as under the prior, can be interpreted as  another manifestation of the \emph{laziness} of the infinite-channel asymptotic regime. 

\subsection{LDP for the rescaled network}

Once an LDP for the covariance tensor $\mathbf{K}^{(L+1,n)}$ has been proved, it is easy to deduce an LDP for the network output. Since the network converges in distribution to a Gaussian tensor, in order to have a meaningful LDP one needs to artificially rescale the network,  as done for the FCNN 
in  \cite{MaPaTo24,MacciTrevisanetal}. 
The starting point is \eqref{vector-rep}, which gives
\[  
\hat{\mathbf{H}}^{(L+1)}_n:= \frac{1}{\sqrt{n}} \mathbf{H}^{(L+1)}\stackrel{\CL}{=} 
\sqrt{\mathbf{I}_{C_{L+1}} \otimes \mathbf{K}^{(L+1,n)}}  \frac{\mathbf{Z}}{\sqrt{n}}
=\frac{1}{\sqrt{n}}
\begin{pmatrix}
\sqrt{\mathbf{K}^{(L+1,n)}} Z_1\\
\dots \\
\sqrt{\mathbf{K}^{(L+1,n)}} Z_{C_{L+1}}
\end{pmatrix}	
\]
where $\bZ=(Z_1^\top,\dots,Z_{C_{L+1}}^\top)^\top$ with 
$Z_c \iid \CN(\mathbf{0},\mathbf{I}_{D_{L+1}})$. 
By Theorem \ref{prop:cnn_lln} $\hat{\mathbf{H}}^{(L+1)}_n \stackrel{\PP}{\to}0$ and one can derive an LDP for such a rescaled process. To state the results, following 
\cite{MacciTrevisanetal}, define for every 
$Q \in \CS^+_{D_{L+1}}$ and $Z
\in \RE^{D_{L+1} }$
  the norm induced by $Q^{-{1/2}}$ (in the generalized sense) as follows:
\[
\|Z\|_Q^2=\begin{cases}
\|Q^{1/2}V\|_2^2=V^\top Q V & \text{if } Z=QV \\
+\infty & \text{if } Z \not\in Im(Q).
\end{cases}
\]
Note that the previous functional is also well-defined when $Ker(Q)\not=0$. Indeed, 
if $V$ and $V'$ are such that $QV=QV'=Z$, then 
$\|Q^{1/2}V\|_2^2=\|Q^{1/2}V'\|_2^2$. 

\begin{proposition}[LDP for the network output]\label{prop:rescalednet}
 Under the same assumptions of 
{\rm Theorem \ref{prop:gs_ldp0}}, 
the sequence
    $(\frac{1}{\sqrt{n}} \mathbf{H}^{(L+1)}\!,\mathbf{K}^{(L+1,n)})_{n \geq 1}$ satisfies a LDP with speed $n$ 
and rate function 
\[
\mathbf{J}(Q,Z)=\frac{1}{2}\sum_{c=1}^{C_{L+1}} \|Z_c\|_{Q}^2
+I_{L+1}(Q) \quad Q \in \CS^+_{D_{L+1}}, \,\,\, Z=(Z_1^\top,\dots,Z_C^\top)^\top \in \RE^{C_{L+1} D_{L+1} }.
\]
In particular, $(\hat{\mathbf{H}}^{(L+1)}_n)_{n \geq 1}$
satisfies a LDP with speed $n$ and rate function 
$Z \mapsto \inf_{Q \in  \CS^+_{D_{L+1}}} \mathbf{J}(Q,Z) $.
\end{proposition}

\begin{proof} Having shown  in  {\rm Theorem \ref{prop:gs_ldp0}} that $(\mathbf{K}^{(L+1,n)})_{n \geq 1}$ satisfies an LDP with rate function $I_{L+1}$, 
    the proof  of the first part follows repeating the same arguments of Example 2.1.1 in 
\cite{MacciTrevisanetal}. 
The second part follows by the
contraction principle. 
\end{proof}

\subsection{Comparison with the fully connected case}
In a standard fully connected neural network, there is no genuine
spatial structure, but this architecture can be recovered as a special case of our general setting  by taking a
single spatial position, i.e.\ $N_\ell=1$ for all $\ell$, and by
identifying the channels with the neurons.
This should be viewed only as a notational embedding of fully
connected networks into the receptive–field framework, 
in particular the limit $C_\ell \to \infty$ becomes the 
usual infinite width limit. 

Taking $T^{(i,\ell)}(z)=\sigma(z)$  and $M_\ell=1$,
 the pre-activation of neuron $c$ in layer $\ell+1$ is
\[
h_{c}^{(\ell+1)}(\bx)
=
\frac{1}{\sqrt{C_\ell}}
\sum_{c'=1}^{C_\ell}
W^{(\ell)}_{c',c}
\sigma\!\left(h^{(\ell)}_{c'}(\bx)\right).
\]

With this choice, the results in \cite{MaPaTo24,Vog24} for the fully connected case are recovered as particular instances of our Theorem \ref{prop:gs_ldp0}. In addition, 
our Assumption \ref{A2_0BIS} on $T^{(i,\ell)}=\sigma$ is weaker than the corresponding hypotheses in \cite{MaPaTo24,Vog24} as well as in \cite{ABH2026,MacciTrevisanetal}, which address the functional setting.

\section{A general covariance structure}\label{Sec:generalstruc}
The Markov property of $\mathbf{K}^{(\ell+1,C_\ell)}$
 plays a central role in establishing the results discussed in the previous section.
Indeed, these results can be viewed as corollaries of analogous statements formulated in a slightly more general framework, which includes, as a particular case, the random covariance structures characteristic of convolutional neural networks.

By Proposition \ref{Prop:conditional_cov},
it is easy to see that    $\mathbf{K}^{(\ell+1,C_\ell)}$ is a Markov sequence.   
In order to describe its Markov structure, let us introduce  the transition kernel on $
\CB(\CS^+_{N_{\ell+1} P}) \times \CS^+_{N_{\ell} P}$ defined by 
\begin{equation}\label{trans.kern}
\nu_{\ell+1,C_\ell}(B|Q):=\PP \left \{  \frac{1}{  C_{\ell}} \sum_{c=1}^{C_{\ell}} G^{(\ell)} ( \sqrt{Q} Z_{c}^{(\ell)}) 
\in B \right \} 
\end{equation}
where $\sqrt{Q}$ is the usual square root of a symmetric and positive semidefinite matrix $Q$
and $Z_{c}^{(\ell)} \stackrel{iid}{\sim} \CN(\mathbf{0},\mathbf{I}_{N_{\ell} P})$ for $c \geq 1$.

\begin{proposition} Under  {\rm \ref{A0}},
   the sequence  of random matrices  $\mathbf{K}^{(\ell+1,C_\ell)}$  with  $\ell = 0,\ldots,L$ is a Markov chain on 
    $\CS^+_{N_{1} P} \times \dots \times \CS^+_{N_{L+1} P}$ with transition kernels 
    $\nu_{\ell+1,C_\ell}$ defined in \eqref{trans.kern}, that is 
    \[
   \PP\{ \mathbf{K}^{(\ell+1,C_\ell)} \in B | \mathbf{K}^{(\ell,C_{\ell-1})}=Q\}=\nu_{\ell+1,C_\ell}(B|Q)
    \]
    for every $B \in \CB(\CS^+_{N_{\ell+1} P})$ and $Q \in  \CS^+_{N_{\ell} P}$ and  
$[\mathbf{K}^{(1,C_0)}_{i,\mu;j,\nu}  ]:= [ \CK_{i,j}^{(1,C_0)} (\bx_\mu,\bx_\nu)] $ specified in 
{\rm Proposition \ref{Prop:conditional_cov}} as initial condition. 
    \end{proposition}

As a generalization of the sequence of covariance matrices 
\((\mathbf{K}^{(1,n)}, \dots, \mathbf{K}^{(L+1,n)})\), 
let us consider a Markov sequence taking values in 
\(\mathcal{S}^+_{D_{L+1}}\), 
with a transition kernel of the same form as in~\eqref{trans.kern}. 
In this more general setting, the specific functions \(G^{(\ell)}\)  in \eqref{def.Gell} 
are replaced by measurable mappings
\[
\mathcal{G}^{(\ell)} : \mathbb{R}^{D_{\ell}} \longrightarrow 
\mathcal{S}^+_{D_{\ell+1}}, 
\qquad \ell = 0, \dots, L.
\]

\begin{definition}\label{def:general_structure} 
For   $\ell = 0,\ldots,L$,  let $Z_{c}^{(\ell)}  \stackrel{iid}{\sim} \CN(\mathbf{0},\mathbf{I}_{D_\ell})$ with $c \geq 1$. 
Given $n \geq 1$, let  $\mathbf{K}^{(1,n)} \allowbreak \dots,\mathbf{K}^{(L+1,n)}$  be a Markov 
    sequence  of random matrices   on 
    $\CS^+_{D_{1}} \times \dots \times \CS^+_{D_{L+1}}$ with transition kernels 
given by
\begin{equation}\label{trans.kern2}
\nu_{\ell+1,n}(B|Q):=\PP \left \{  \frac{1}{  C_{\ell}(n)} \sum_{c=1}^{C_{\ell}(n)} \CG^{(\ell)} ( \sqrt{Q} Z_{c}^{(\ell)} ) 
\in B \right \} ,  \, B \in \CB(\CS^+_{D_{\ell+1}}), \, Q \in  \CS^+_{D_{\ell}}
\end{equation}
that is $\PP\{ \mathbf{K}^{(\ell+1,n)} \in B | \mathbf{K}^{(\ell,n)}=Q\}=\nu_{\ell+1,C_\ell}(B|Q)$
and initial condition 
    $\mathbf{K}^{(1,n)}=\mathbf{K}^{(1)}$, $\mathbf{K}^{(1)}$ being a deterministic fixed matrix
    {\rm(}independent from $n \geq 1$ {\rm)}.
\end{definition}

Recalling that    $\| \cdot \|_2$ denotes the Euclidean norm,  
    with a slight abuse of notation, we use the same symbol  also for the corresponding 
     matrix norm $\|Q\|_2=\sup_{x: \|x\|_2=1} \|Qx\|_2$. 
    We shall also need the Frobenius norm of a matrix,  
    defined as 
$\| Q\|_F:=\sqrt{\tr(Q^\top Q)}$.

\begin{remark}[Matrix norms]\label{matrixnorm}
Let us recall some useful properties of these matrix norms. 
    \begin{itemize}
    \item[(a)] for any $D \times D$ matrix $Q$ and any vector $x \in \R^D$, it holds that $\|Qx\|_2 \leq \|Q\|_2 \, \|x\|_2$;
  \item[(b)]  for any $Q \in \psd{D}$ one has 
$\|\sqrt{Q}\|_2 = \sqrt{\|Q\|_2}$;    \item[(c)] for any matrix $Q$, one has $\| Q\|_F^2=\sum_{km}Q_{km}^2$;
    \item[(d)] 
 $\|Q\|_2\leq \| Q\|_F \leq D\|Q\|_2$. 
\end{itemize}
\end{remark}

We shall consider the following  assumptions.

\begin{assumption}\label{A2} 
For every $\ell=0,\dots,L$, 
the  functions  $\CG^{(\ell)}$ are continuous and exponentially  bounded with degree $r <2$, i.e. 
\[
\| \CG^{(\ell)}(z) \|_F  \leq Ae^{B\|z\|_2^{r}} \quad  \forall z \in  \RE^{D_{\ell} }.
\]
\end{assumption}

\begin{assumption}\label{A2bis} For every $\ell=0,\dots,L$, $\CG^{(\ell)}$ is continuous and satisfies the following condition: 
for every $\varepsilon>0$ there is $C_\varepsilon>0$  such that for every $z,z' \in \RE^{D_\ell}$
\begin{equation}\label{A2bisbis}
     \left\|\CG^{(\ell)}(z)-\CG^{(\ell)}(z')\right\|_F
 \leq C_\varepsilon \left[1+\|z-z'\|_2 (\|z\|_2+\|z'\|_2) \right] 
 +\varepsilon(\|z\|_2^2+\|z'\|_2^2) .
\end{equation}
In particular, the $\CG^{(\ell)}$'s are  polynomially bounded with degree $2$, i.e.  
\begin{equation}\label{poly2G}
    \| \CG^{(\ell)}(z) \|_F  \leq A(1+ \|z\|^2_2) \quad  \forall z \in  \RE^{D_{\ell} }.
    \end{equation}
\end{assumption}

We stress the fact that the $\CG^{(\ell)}$ are not assumed to have the specific form  of the $G^{(\ell)}$ given  in \eqref{def.Gell}. On the other hand, under 
  \ref{A2_0} --\ref{A2_0BIS}, respectively-- 
the function $G^{(\ell)}$ defined by \eqref{def.Gell}
 satisfies  \ref{A2} --\ref{A2bis}, respectively-- as stated in the next Lemma.

 \begin{lemma}\label{lemma:fromCNNtoGeneral}  Let 
the function $G^{(\ell)}$ be defined by \eqref{def.Gell}. 
\begin{enumerate} 
 \item If 
{\rm\ref{A2_0}}  holds then $\CG^{(\ell)}=G^{(\ell)}$ 
 satisfies  {\rm\ref{A2}}. 
 \item  If {\rm\ref{A2_0BIS}} holds then $\CG^{(\ell)}=G^{(\ell)}$  satisfies {\rm\ref{A2bis}}. 
\end{enumerate}
\end{lemma}
 The proof is reported in Appendix \ref{AppendixA}.

We shall use several times  the following simple bound, 
which is a  consequence of Remark \ref{matrixnorm}.

\begin{lemma}\label{ineqNormaG} 
If {\rm \ref{A2}}  holds, then 
for every $z \in \RE^{D_\ell}$
\begin{equation*}
\sup_{Q \in \CS^+_{D_{\ell+1}}: \|Q\|_F \leq M }   \left\lnorm 
    \CG^{(\ell)} \left(
    \sqrt{Q} z
    \right)\right\rnorm_2
    \leq
    A  \left( 1 + M \left\|z \right\|_2^{2}
    \right).
\end{equation*}
If {\rm\ref{A2bis}} holds, then 
\begin{equation*}
\sup_{Q \in \CS^+_{D_{\ell+1}}: \|Q\|_F \leq M }   \left\lnorm 
    \CG^{(\ell)} \left(
    \sqrt{Q} z
    \right)\right\rnorm_2
    \leq
    A  e^{B M^{r/2} \left\|z \right\|_2^{r}}.
\end{equation*}
\end{lemma}

\begin{remark}\label{continuitySqrt} 
 The principal square root map on the cone of symmetric positive semidefinite matrices $\psd{D}$ is Hölder continuous of order $1/2$ with respect to the operator norm, see 
 e.g. Theorem V.1.9 in \cite{bhatia97}. Hence, 
if $Q_{1,n} \to Q$ in  $\CS^+_{D}$  (in any matrix norm), then $\sqrt{Q_{1,n}} \to \sqrt{Q}$.  
\end{remark}

\section{Proofs}\label{Sec:proofs}

The first subsection is devoted to the proof of Theorem~\ref{prop:cnn_lln}, whereas the remaining subsections develop the proof of Theorem~\ref{prop:gs_ldp0}.

\subsection{Proof of Theorem \ref{prop:cnn_lln}}
\label{Sec:Proofcnnlln}
 The proof of Theorem \ref{prop:cnn_lln} relies on the following two lemmas, whose proofs are provided in Appnedix \ref{AppendixA}.

\begin{lemma}\label{lemma:triangle_lln} 
     Assume {\rm \ref{A2}}.
    Let $\{Q_n\}_n\subset\CS_{D_\ell}^+$ such that $Q_n\to Q\in\CS_{D_\ell}^+$, and take $Z,Z_{1},Z_2,\ldots  \iid \CN(\mathbf{0},\id_{D_\ell})$. Then
    \begin{equation*}
        \frac{1}{C_\ell(n)}\sum_{c=1}^{C_\ell(n)}\CG^{(\ell)}\left(\sqrt{Q_n}Z_c \right) \stackrel{\PP}{\rightarrow} \E\left[\CG^{(\ell)}\left(\sqrt{Q}Z\right)\right].
    \end{equation*}
\end{lemma}

\begin{lemma} \label{lemma:bridge-lln}
Let $\{(X_n,Y_n)\}_n$ be a sequence of random variables taking values in two Polish  spaces $\XX$ and $\YY$, respectively. 
Let $\nu_n(dy|x)$ be a version of  the conditional distribution   $\mathcal{L}(Y_n \mid X_n = x)$.  Suppose that:
(1) There exists a constant $x_0\in\XX$ such that $X_n \xrightarrow{\PP} x_0$.
   (2) There exists a constant $y_0 \in \YY$ such that for any sequence $\{x_n\}_{n} \subseteq \XX$ with $x_n \to x_0$,
   $\nu_n(dy|x_n)$ converges weakly to $\delta_{y_0}$.
Then, it follows that $Y_n$ converges in probability to $y_0$.
\end{lemma}

Having established these two results, we can prove the main result of this subsection.

\begin{proposition}[General covariance concentration] 
\label{prop:gs_lln}
Let Assumptions {\rm\ref{A1}} and {\rm\ref{A2}}  hold. 
Let $\mathbf{K}^{(1,n)}, \ldots, \mathbf{K}^{(L+1,n)}$ be as in {\rm Definition \ref{def:general_structure}}. Then, 
as $n \to +\infty$, 
\begin{equation*}
    \left(
     \mathbf{K}^{(2,n)},\ldots, \mathbf{K}^{(L+1,n)}
    \right) 
    \stackrel{\PP}{\to}
    \left(
     \mathbf{K}^{(2)},\ldots, \mathbf{K}^{(L+1)}
    \right),
\end{equation*}
where the limit is defined recursively for $\ell = 1,\ldots,L$ as 
\begin{equation*}
     \mathbf{K}^{(\ell+1)} = 
  \E\left[
    \CG^{(\ell)} \left( \sqrt{ \mathbf{K}^{(\ell)}} Z^{(\ell)} \right)
    \right]
\end{equation*}
 with  $Z^{(\ell)} \sim\CN\left( 0,\id_{D_\ell} \right)$.
\end{proposition}

\begin{proof}
The proof proceeds by induction, where the base case, $\mathbf{K}^{(2,n)} \to \mathbf{K}^{(2)}$ is simply an application of the law of large numbers, since $\mathbf{K}^{(1)}$ is deterministic. 
Thereafter, for $\ell = 2, \ldots, L$, suppose that, as $n \to \infty$, the first $\ell$ elements of the Markov chain converge in probability to the limit stated in the proposition. This, in particular, implies that $\mathbf{K}^{(\ell,n)} \xrightarrow{\mathbb{P}} \mathbf{K}^{(\ell)}$, meaning that the first condition of Lemma~\ref{lemma:bridge-lln} holds.
Now consider $\nu_{\ell+1,n}$ as in Definition \ref{def:general_structure} and take any sequence $\{ Q_n\}_n \subset \psd{D}$ such that $Q_n\to Q:=\mathbf{K}^{(\ell)}$. Observe that $\nu_{\ell+1,n}(\cdot\mid Q_n)$ is by definition the law of the random variable 
\[
  \frac{1}{  C_{\ell}(n)} \sum_{c=1}^{C_{\ell}(n)} \CG^{(\ell)} ( \sqrt{Q_{n} } Z_{c}^{(\ell)} ) .
  \]
At this point, as $n\to\infty$ an application of Lemma \ref{lemma:triangle_lln} leads to
\begin{equation*}
   \frac{1}{  C_{\ell}(n)} \sum_{c=1}^{C_{\ell}(n)} \CG^{(\ell)} ( \sqrt{Q_{n} } Z_{c}^{(\ell)} )      \convp 
 \E\left[ \CG^{(\ell)}\left(\sqrt{\mathbf{K}^{(\ell)}} Z^{(\ell)}\right) \right].
\end{equation*} 
The consequence is that the second hypothesis of Lemma~\ref{lemma:bridge-lln}  also holds, which implies that $ \mathbf{K}^{(\ell+1,n)} \xrightarrow{\mathbb{P}}  \mathbf{K}^{(\ell+1)}$, thereby concluding the inductive step and hence the proof.
\end{proof}

In the light of Lemma \ref{lemma:fromCNNtoGeneral}, Theorem \ref{prop:cnn_lln}
follows from Proposition \ref{prop:gs_lln}.

\subsection{Conditional large deviation principle}

In order to prove a full LDP for the sequence of $\{ (\mathbf{K}^{(2,n)}, \ldots, \mathbf{K}^{(L+1,n)})\}_n$, 
we shall apply  a "conditional large deviation principle" derived in~\cite{Cha97}. 

\begin{definition}[conditional LDP continuity condition, \cite{Cha97}] 
\label{def:ldp_cont_cond}
    Let $\XX_1$ and $\XX_2$ be two Polish spaces. A sequence of transition kernels $\{\nu_n^{(2|1)} : \CB(\XX_2) \times  \XX_1  \to [0,1]\}_n$ is said to satisfy the conditional LDP continuity condition with rate function $\CI_{2|1}(\cdot|\cdot)$ if:
    \begin{enumerate}
        \item For each $x_1 \in \XX_1$, $\CI_{2|1}(\cdot|x_1)$ is a good rate function on $\XX_2$.
        \item For each $x_1 \in \XX_1$ and each sequence $x_{1,n} \to x_1$, $\{\nu_n^{(2|1)}(\cdot|x_{1,n})\}_{n \ge 1}$ satisfies a large deviation principle  on $\XX_2$ with rate function $\CI_{2|1}(\cdot|x_1)$.
        \item The mapping $(x_1, x_2) \mapsto \CI_{2|1}(x_2|x_1)$ is lower semi-continuous.
    \end{enumerate}
\end{definition}

Recall also that 
a sequence of measures $\{\mu_n\}_n$ satisfies a weak large deviation principle with rate function $I$ 
if the upper bound  \eqref{eq:ldp_upper}  is true only for compact sets and not for generic closed sets. 

Given these definitions, the result from~\cite{Cha97} reads as follows.

\begin{proposition}[Theorem~2.3, \cite{Cha97}]\label{thm:joint_LDP_general}
  Let $\XX_1$ and $\XX_2$ be two Polish spaces. 
   Let $\{\nu_n^{(1)}\}_n$ be a sequence of probability measures on $\CB(\XX_1)$, satisfying a LDP with good rate function $I_1$. 
  Let  $\{\nu_n^{(2|1)}\}_n$ be a  sequence of transition kernels on  $\CB(\XX_2)  \times  \XX_1$
   satisfying a conditional LDP continuity condition with rate function $\CI_{2|1}(\cdot|\cdot)$. 
Then
     the sequence of measures $\{\mu_n\}_n$ defined by 
    \[
    \mu_n (A \times B) = \int_A \nu_n^{(2|1)}(B|x) \,\nu_n^{(1)}(dx) \qquad A \in \CB(\XX_1) \, \, B \in \CB(\XX_2)
    \]
    satisfies a weak LDP with rate function $\CI(x_1,x_2)= \CI_1(x_1) + \CI_{2|1}(x_2|x_1)$.
\end{proposition}

The key point is to show that each sequence of kernels in Definition \ref{def:general_structure} satisfies the LDP continuity condition above. Following the strategy used for the fully connected neural network in \cite{Vog24},
this will be obtained by combining exponential equivalence and Cramer's theorem. 
See also \cite{MacciTrevisanetal}.

For $Q_1 \in \RE^{D_{\ell} \times D_\ell}$ and  $Q_0 \in  \sym{D_{\ell+1}}$ define 
\[
   M_{\ell}(Q_0|Q_1):=    
        \E \left[\exp 
        \left(
         \tr(Q_0^\top \CG^{(\ell)}(\sqrt{Q_1}Z))
        \right)\right]
\]
where $Z \sim \CN_{D_\ell}(\mathbf{0},\mathbf{I}_{D_\ell})$
and 
\[
\CD_{\ell,Q_1}:=\{Q_0\in  \RE^{D_{\ell} \times D_\ell}: \, M_\ell(Q_0|Q_1)<+\infty \}.
\]

Note that for  $Q_1 \in \psd{D_{\ell}}$ and $Q_2 \in \psd{D_{\ell+1}}$, 
the function 
 $I_{\ell}(Q_2|Q_1)$  defined in  \eqref{eq:gs_rate0}
with $\CG^{(\ell)}$ in place of  $G^{(\ell)}$ 
 can be written as 
 \begin{equation}\label{eq:gs_rate}
      I_\ell(Q_2|Q_1)=
 \sup_{Q_0 \in  \sym{D_{\ell+1}}} 
 [\tr(Q_0^\top Q_2)-\log(M_\ell(Q_0|Q_1))].
  \end{equation}
Moreover, it is easy to check that 
\begin{equation}\label{eq.Ieqv.def}
\sup_{Q_0 \in  \sym{D_{\ell+1}}} 
 [\tr(Q_0^\top Q_2)-\log(M_\ell(Q_0|Q_1))]=
 \sup_{Q_0 \in\RE^{D_{\ell+1} \times D_{\ell+1}}} 
 [\tr(Q_0^\top Q_2)-\log(M_\ell(Q_0|Q_1))].
\end{equation}

\begin{lemma}\label{Lemma_bound_lip_onG}
 Let  $Q_1,Q_{1,n} \in \psd{D_{\ell}}$ for every $n$ and assume that  
  $Q_{1,n} \to Q_1$ as $n \to \infty$. 
Assume also that  {\rm \ref{A2bis}} holds true. Then, 
  for every $\eta>0$ 
there is  $ \bar n$ and $\bar C_\eta<+\infty$ such that 
\begin{equation}
\label{eq:increment_bound2}
\left\|\CG^{(\ell)}\left(\sqrt{Q_1}z \right)- \CG^{(\ell)} \left(\sqrt{Q_{1,n}} z\right) \right\|_F \le \bar C_\eta +\eta \|z\|^2 
\qquad \forall n \geq \bar n \qquad \forall z \in \RE^{D_\ell}.
\end{equation}
 \end{lemma}

 \begin{proof}
     One can assume that $\|\sqrt{Q_{1,n}}\|\leq M$ and $\|\sqrt{Q_{1}}\|\leq M$. Using 
 \ref{A2bis}, one gets 
\begin{equation*}
\label{eq:increment_bound}
\begin{split}
   & \left\|\CG^{(\ell)}\left(\sqrt{Q_1}z\right)- \CG^{(\ell)} \left(\sqrt{Q_{1,n}}z\right) \right\|_F \\ 
   &\le C_\varepsilon 
\Big[ 1+
\left\| \sqrt{Q_1}-\!\sqrt{Q_{1,n}}\right\|_2  
\left(\left\|\sqrt{Q_1}\right\|_2  + \left\|\sqrt{Q_{1,n}}\right\|_2  \right)
\|z\|_2^2 \Big]
\\
&
+ \varepsilon\left(\left\|\sqrt{Q_1}\right\|_2^2 + \left\|\sqrt{Q_{1,n}}\right\|_2^2 \right)\|z\|_2^2
\le 
C_\varepsilon + 2M\Big ( C_\varepsilon  \|\sqrt{Q_1}-\sqrt{Q_{1,n}}\|_2 +M \varepsilon \Big ) \|z\|_2^2  \\
\end{split}
\end{equation*}
where $\|\sqrt{Q_1}-\sqrt{Q_{1,n}}\|_2 \to 0$. 
Hence, for every $\eta>0$ there exists $\varepsilon>0$ and 
 $ \bar n(\varepsilon)$  such that 
$2M[C_\varepsilon  \|\sqrt{Q_1}-\sqrt{Q_{1,n}}\|_2 +M \varepsilon] \leq \eta$ for every $n \geq \bar n$.
 \end{proof}

\begin{lemma}\label{Lemma_genMom} Let  $Q_1 \in \psd{D_{\ell}}$ and  
assume  {\rm\ref{A2bis}}.
Then:
\begin{itemize}
    \item[(a)] $\{Q_0: \|Q_0\|_F < \frac{1}{2A\|Q_1\|_2 }  \} \subset \CD_{\ell,Q_1}$. 
In particular, $\mathbf{0} \in \overset{\circ}{\CD}_{\ell,Q_1} $.
    \item[(b)] For every $Q_0 \in  \overset{\circ}{\CD}_{\ell,Q_1}$ if $Q_{1,n} \to Q_1$ then 
$\lim_n M(Q_0|Q_{1,n})=M(Q_0|Q_{1})$.
\item[(c)] For every $Q_0 \in {\CD}_{\ell,Q_1}$, 
there is $Q_{0,n} \in \overset{\circ}{\CD}_{\ell,Q_1}$ 
such that 
\[
\lim_{n \to +\infty} \log(M(Q_{0,n}|Q_1)) \leq 
\log(M(Q_0|Q_1)).
\]
\end{itemize}
\end{lemma}

\begin{proof} Using \ref{poly2G}, one can write 
\[
 |\tr(Q_0^\top \CG^{(\ell)}(\sqrt{Q_1}Z))|  \leq \|Q_0\|_F \|\CG^{(\ell)}(\sqrt{Q_1}Z))\|_F 
\leq \|Q_0\|_F  A(1+\|\sqrt{Q_1}\|_2^2\|Z\|_2^2). 
\]
Point (a) follows since 
$\E[e^{t\|Z\|_2^2}]<+\infty$
for every $t<1/2$. To prove (b), note that 
\begin{equation}
\begin{split}
\tr\left(Q_0^\top\CG^{(\ell)}\left(\sqrt{Q_{1,n}}Z\right)\right)
& \leq  
\tr\left(Q_0^\top\CG^{(\ell)}\left(\sqrt{Q_{1}}Z\right)\right)   \\
& +
\|Q_0\|_F  \left\|\CG^{(\ell)}\left(\sqrt{Q_1}Z\right)- \CG^{(\ell)} \left(\sqrt{Q_{1,n}} Z\right) \right\|_F. \\
\end{split}
\end{equation}
By \eqref{eq:increment_bound2}, one can assume
 that 
$
\left\|\CG^{(\ell)}\left(\sqrt{Q_1}Z \right)- \CG^{(\ell)} \left(\sqrt{Q_{1,n}} Z\right) \right\|_F \le \bar C_\eta +\eta \|Z\|^2 $ 
for every $n \geq \bar n$. 
Using the convexity of the exponential function, for every $\delta>0$, it holds that
$e^{a+b} \le 
(1+\delta)^{-1}\, e^{(1+\delta)a}
+ {\delta}{(1+\delta)^{-1}}\  e^{\frac{(1+\delta)}{\delta}\, b}$,
and hence
\begin{align*}
\exp\left( \tr\!\left(Q_0^\top \CG^{(\ell)}\!\left(\sqrt{Q_{1,n}}Z\right)\right) \right)
&\le \frac{1}{1+\delta} \exp\!\left( \tr\!\left((1+\delta) Q_0^\top \CG^{(\ell)}\!\left(\sqrt{Q_{1}}Z\right)\right) \right)
\\
&+
\frac{\delta}{(1+\delta)} \exp\left( \frac{1+\delta}{\delta}  \|Q_0\|_F (\bar C_\eta+\eta \|Z\|_2^2 ) \right) =:H(Z).
\end{align*} 
By point (a), being $Q_0$ in the interior of $\CD_{\ell,Q_1}$, 
one can choose $\delta>0$ such that $(1+\delta) Q_0 \in 
\CD_{\ell,Q_1}$. Hence
\[
\E\left[\exp\!\left( \tr\!\left((1+\delta) Q_0^\top \CG^{(\ell)}\!\left(\sqrt{Q_{1}}Z\right)\right) \right)\right]<+\infty.
\]
Moreover, given such $\delta$, one can choose $\eta$ such that $\eta\frac{1+\delta}{\delta}  \|Q_0\|_F  <1/2$, so that 
\[
\E\left[\exp\left( \eta \frac{1+\delta}{\delta}  \|Q_0\|_F\|Z\|_2^2 \right)\right]<+\infty.
\]
 This shows that 
 with $\E[H(Z)]<+\infty$.
 Since by continuity 
\[
\tr(Q_0^\top\CG^{(\ell)}(\sqrt{Q_{1,n}}Z)) \to \tr(Q_0^\top\CG^{(\ell)}(\sqrt{Q_{1}}Z))\]
almost surely, point (b)
follows by dominated convergence theorem. 
It remains to prove (c). 
It is well-known and easy to check that $Q_0 \mapsto \log(M_\ell(Q_0|Q_1))$ is convex in $\CD_{\ell,Q_1}$, hence if
$Q_0 \in \CD_{\ell,Q_1}$ then $Q_{0,n}=(1-1/n)Q_0=(1-1/n)Q_0+1/n\mathbf{0} \in \overset{\circ} {\CD}_{\ell,Q_1}$ (by part (a) and convexity) and 
$
 \log( M_\ell(Q_{0,n}|Q_1))  \leq (1-1/n)  \log( M_\ell(Q_0|Q_1) ) +1/n  \log(M_\ell(\mathbf{0} |Q_1))$, which is $ \log( M_\ell(Q_{0,n}|Q_1)) \leq  
 (1-1/n)  \log( M_\ell(Q_0|Q_1) )$.
Taking the limit one has  (c). 
\end{proof}

\begin{lemma}\label{lemma:r=2lsc}
Assume {\rm \ref{A2bis}}. 
If  $(Q_{1,n},Q_{2,n}) \to (Q_1,Q_2)$ in $\psd{D_{\ell}}\times \sym{D_{\ell+1}}$, then 
 \[
 \liminf_{n \to +\infty} I_\ell(Q_{2,n}|Q_{1,n}) \geq I_\ell(Q_{2}|Q_1).
 \]
\end{lemma}

\begin{proof}
    Let $Q_0 \in \sym{D_{\ell+1}}$, then 
$
I_\ell(Q_{2,n}|Q_{1,n})
\geq \tr(Q_0^\top Q_{2,n})-\log(M(Q_0|Q_{1,n}))$.
Now $\tr(Q_0^\top Q_{2,n}) \to \tr(Q_0^\top Q_{2})$ for any $Q_0$
and 
 $\log(M(Q_0|Q_{1,n})) \to \log(M(Q_0|Q_{1})) $ if
$Q_0 \in  \overset{\circ}{\CD}_{\ell,Q_1}$,  by (b) in Lemma \ref{Lemma_genMom}. Hence
\[
\liminf_n I_\ell(Q_{2,n}|Q_{1,n})
\geq  \sup_{Q_0 \in  \overset{\circ}  \CD_{\ell,Q_1} }
[\tr(Q_0^\top Q_2)-\log(M(Q_0|Q_{1}))].
\]
Using (c) in Lemma \ref{Lemma_genMom}, one can replace   $\overset{\circ} \CD_{\ell,Q_1}$ 
 by  $\CD_{\ell,Q_1}$  in the previous expression and this gives the claim. 
\end{proof}

\begin{lemma}\label{Lemma:r=2cLDP} Let  $Q_1,Q_{1,n} \in \psd{D_{\ell}}$ for every $n$ and assume that  
  $Q_{1,n} \to Q_1$ as $n \to \infty$. 
Assume also that
{\rm \ref{A2bis}} holds true.
Let  $Z_k \iid \CN(\mathbf{0},\id_{D_{\ell}})$ and set
\[
S_n=\frac{1}{C_\ell(n)} \sum_{k = 1}^{C_\ell(n)}\CG^{(\ell)} \left(\sqrt{Q_{1,n}} Z_{k}\right). 
\]
Then 
$S_n$ satifies a LDP  on  $\psd{D_{\ell+1}}$ 
with good rate function $Q_2 \mapsto  \alpha_\ell I_\ell(Q_2|Q_1)$ and 
 speed $n$. 
\end{lemma}
\begin{proof}
Since the LDP is defined via upper and lower bounds along the full sequence, the same bounds automatically hold along any subsequence.
Given the arbitrariness of $Q_{1,n}$, we may view $S_n$ as a subsequence; hence, it suffices to prove an LDP with rate $I_\ell$ for the special case in which $S_n$ is defined with $n$ in place of $C_\ell(n)=n$. The resulting rate function $\alpha_\ell I_\ell$ arises from the modification of the speed together with the hypothesis $C_\ell(n)/n \to \alpha_\ell$:
\[
\frac{1}{n} \log \PP(S_n \in A)
=\frac{C_\ell(n)}{n} \frac{1}{C_\ell(n)} \log \PP(S_n \in A)
\sim \alpha_\ell  \frac{1}{C_\ell(n)} \log \PP(S_n \in A).
\]
Set
$
\tilde  S_n=\frac{1}{n}  \sum_{k = 1}^n \CG^{(\ell)} \left(\sqrt{Q_1} Z_{k}\right)$. 
Using Lemma \ref{Lemma_genMom} (a) and 
the abstract Cramer theorem  (see, e.g. 
 Corollary 6.1.6
in \cite{dembo98}), 
$\tilde S_n$ satifies a (full) LDP on $\RE^{D_\ell\times D_\ell}$
with good rate function 
$
Q_2 \mapsto \sup_{Q_0 \in   \RE^{D_{\ell} \times D_\ell} } 
 [\tr(Q_0^\top Q_2)-\log(M_\ell(Q_0|Q_1))]$.
which is equal to 
$I_\ell(Q_2|Q_1)$ by \eqref{eq.Ieqv.def} when $Q_2$
is in $\psd{D_{\ell+1}}$.
  We now show  that $S_n$ and $\tilde S_n$
are exponentially equivalent, that is (see e.g. Definition 4.2.10
in \cite{dembo98}): for every $\delta>0$
\[
\lim_{n \to +\infty } \frac{1}{n} \log \left (\PP\{ \|S_n -\tilde S_n\|_F > \delta\} \right) =-\infty . 
\]
Markov inequality and triangular inequality give
\begin{equation*}
    \begin{split}
    & \PP\left(  \left\| \frac{1}{n} \sum_{k = 1}^n \left(\CG^{(\ell)} \left(\sqrt{Q_1} Z_{k} \right) - \CG^{(\ell)} \left(\sqrt{Q_{1,n}} Z_{k}\right)\right)\right\| _F>\delta \right)  
    \\
    &\le
    \exp(- n t  \delta) 
\left (\E\left[ \exp\left( 
    t \left\| \CG^{(\ell)}\left(\sqrt{Q_1}Z_{1}\right)- \CG^{(\ell)} \left(\sqrt{Q_{1,n}} Z_{1}\right) \right\|_F
    \right)\right]\right)^n.
\end{split}
\end{equation*}
For $\eta<1/2$, 
by Lemma \ref{Lemma_bound_lip_onG},  $\|\CG^{(\ell)}(\sqrt{Q_1}Z_{1})- \CG^{(\ell)} (\sqrt{Q_{1,n}} Z_{1})\|_F
\le 
\bar C_\eta+\eta \|Z_1\|^2$ for $n \geq \bar n$. 
Hence, by the dominated convergence theorem we get, for every $t>0$, 
\[
\lim_{n\to\infty}\log\E\left[\exp\left( t \left\| \CG^{(\ell)}\left(\sqrt{Q_1}Z_{1}\right)- \CG^{(\ell)} \left(\sqrt{Q_{1,n}} Z_{1}\right) \right\|_F\right)\right] = 0.
\]
At this point having
\[
\limsup_{n\to\infty} \frac{1}{n}\log\PP\left(  \left\| \frac{1}{n} \sum_{k = 1}^n \left(\CG^{(\ell)} \left(\sqrt{Q_1} Z_{k} \right) - \CG^{(\ell)} \left(\sqrt{Q_{1,n}} Z_{k}\right)\right)\right\| _F>\delta \right)  
\le 
-t\delta,
\]
we get the exponential equivalence by taking the limit as $t\to\infty$.

By Theorem 4.2.13  in \cite{dembo98}  also $S_n$ 
satifies an LDP with good rate function $I_\ell$. 
Since  $\psd{D_{\ell+1}}$ is closed in $\RE^{D_\ell \times D_\ell}$ the thesis follows by  Lemma 4.1.5 in \cite{dembo98}.  
\end{proof}

\begin{proposition}\label{prop:ldp-continuity}
    For every $\ell = 0,\ldots,L$ let 
    $\{\nu_{\ell+1,n}:\CB(\psd{D_{\ell+1}}) \times \psd{D_\ell}  \to [0,1]\}_n$ be as in {\rm Definition \ref{def:general_structure}}
    and let the 
     Assumptions {\rm \ref{A1}} and {\rm \ref{A2bis}}  be valid. 
    Then $\nu_{\ell+1,n}$ satisfies the conditional LDP continuity condition 
  on $\psd{D_{\ell}} \times \psd{D_{\ell+1}}$  
    with  speed $n$ and good  rate function $\alpha_\ell I_\ell(Q_2|Q_1)$, where 
    $I_\ell$ is defined 
    in \eqref{eq:gs_rate}. 
\end{proposition}
\begin{proof} 
Points   (1) and (2) of Definition \ref{def:ldp_cont_cond} follow by 
Lemma \ref{Lemma:r=2cLDP}. Point (3) by 
 Lemma \ref{lemma:r=2lsc}.  
\end{proof}

Upgrading the weak LDP in Proposition \ref{thm:joint_LDP_general} to a full LDP is not always straightforward, as there exist sequences of measures satisfying a weak but not a full LDP (see \cite{dembo98}). 
Nevertheless, if $\{\mu_n\}_n$ satisfies a WLDP with rate function $I$ and is exponentially tight, then it obeys a full LDP with good rate function $I$ (see Theorem 2.3(iii) in \cite{rezakhanlou2017}).  

The next section establishes exponential tightness for the sequence of random covariance matrices.

\subsection{Exponential Tightness}
We start with a simple preliminary result. 

\begin{lemma}\label{lemmatight}
 Let Assumptions {\rm \ref{A1}} and {\rm\ref{A2bis}} hold.  Let $Z_k\iid\CN(\mathbf{0},\id_D)$
and let $S \subset \psd{D}$ be compact, then  for every $\epsilon>0$, 
there is $R=R(D_\ell,S,\CG^{(\ell)},\eps)$ such that 
 \[
\sup_{Q \in S}    \PP\left( \left\| \frac{1}{C_\ell(n)} \sum_{k = 1}^{C_\ell(n)} \CG^{(\ell)} \left(\sqrt{Q} Z_{k}\right) \right\| _F>R \right) 
    \leq e^{-n \eps}
 \]
\end{lemma}

\proof
Without loss of generality, assume that $C_\ell(n)=n$.
Assume that $Q \in S$, Markov inequality gives
\begin{align*}
    \PP\left( \left\| \frac{1}{n} \sum_{k = 1}^n \CG^{(\ell)} \left(\sqrt{Q} Z_{k}\right) \right\| _F>R \right) 
    &\le
    \exp(- n t  R) 
    \E\left[ \exp\left( 
    t \left\| \sum_{k = 1}^n \CG^{(\ell)}\left(\sqrt{Q}Z_{k}\right) \right\|_F
    \right)\right]
\end{align*}
for every $t>0$ such that the expectation is finite. 
Now by the triangle inequality we get 
\begin{align*}
    \E\left[ \exp\left( 
    t \left\| \sum_{k = 1}^n \CG^{(\ell)}  \left(\sqrt{Q}Z_{k}\right) \right\|_F
    \right)\right]
    &\le
    \E\left[ \exp\left( 
    t \sum_{k = 1}^n \left\| \CG^{(\ell)}  \left(\sqrt{Q}Z_{k} \right) \right\|_F
    \right)\right]
    \\
    &=
   \E\left[ 
    \exp\left( 
    t  \left\| \CG^{(\ell)}  \left(\sqrt{Q}Z_1 \right) \right\|_F
    \right) \right]^n
\end{align*}
since $Z_{1},\ldots,Z_{n}$ are independent and identically distributed. 
Let $M > 0$ such that $\|Q\|_2 \le M$ for all $Q \in S$. 
Using once again  Lemma \ref{ineqNormaG}
and $\|\cdot\|_F \leq D \|\cdot\|_2$,  one gets
\begin{align*}
\sup_{  Q \in S}   \E\left[ 
    \exp\left( 
    t  \left\| \CG^{(\ell)} \left(\sqrt{Q}Z_1 \right) \right\|_F
    \right) \right]     \le
    \E\left[ 
    \exp\left( 
    t  AD_\ell \left( 1 + M \left\|Z_1\right\|^{2} \right)
    \right) \right]<+\infty
\end{align*}
for $t$ small enough since $Z$ is Gaussian. 
Hence
\begin{align*}
&\sup_{  Q \in S}    \PP\left( \left\| \frac{1}{n}  \sum_{k = 1}^n \CG^{(\ell)} \left(\sqrt{Q} Z_{k}\right) \right\| _F>R \right) 
 \\& 
\leq 
    \exp\left( 
    -n t R + n\log 
    \E\left[ \exp\left( 
    t AD_\ell  \left( 1 + M \left\|Z_1\right\|^{2} \right)
    \right) \right]
    \right) 
   \le
    e^{-n\eps}
\end{align*}
by taking $R$ big enough to have 
$    -t R + \log 
    \E\left[ \exp\left( 
    t D_\ell  A \left( 1 + M \left\|Z_1\right\|^{2} \right)
    \right) \right] \le -\eps
$.
\endproof

We now show that the entire sequence of random matrices in Definition \ref{def:general_structure} is exponentially tight.
\begin{proposition}[Exponential tightness]
\label{prop:exptight-K}
    Let $\{ (\mathbf{K}^{(2,n)}, \ldots, \mathbf{K}^{(L+1,n)})\}_n$ be as in {\rm Definition \ref{def:general_structure}}.
       Let {\rm Assumptions  \ref{A1}} and {\rm\ref{A2bis}}  hold. 
    Then the aforementioned sequence is exponentially tight as a joint collection of matrices, i.e. for every $\eps > 0$ there exist compact sets $S_2,\ldots, S_{L+1} \subset \CS_D^+$ such that
    \begin{equation*}
        \label{eq:exp_tight_kernels}
        \limsup_{n \to \infty} \frac{1}{n} \log 
        \PP\left( \left( \mathbf{K}^{(2,n)}, \ldots, \mathbf{K}^{(L+1,n)}\right) \notin 
        S_2 \times \cdots \times S_{L+1}\right)
        \le - \eps
    \end{equation*}
\end{proposition}
\proof
Let $\nu_{2,n},\ldots,\nu_{L+1,n}$ be the sequence of kernels as in Definition \ref{def:general_structure}. The proof proceeds by induction, starting from the base case $\left(\mathbf{K}^{(2,n)}, \mathbf{K}^{(3,n)} \right)$. 
Since $ \mathbf{K}^{(1)}$ is deterministic and fixed, we can take $S=\{Q =  \mathbf{K}^{(1)}\}$ to conclude that the sequence $ \mathbf{K}^{(2,n)}$ is exponentially tight
by Lemma \ref{lemmatight}. 
Then, for any fixed $\eps > 0$ there exist a compact set $S_2 \subset \CS_D^+$ (a suitable closed ball $S_2=\{Q:\|Q\|_F\leq R_2\}$ derived from Lemma \ref{lemmatight})  
such that 
$\PP\left(  \mathbf{K}^{(2,n)} \notin S_2  \right) \le 
    e^{-n\eps}$.
Define the set $S_3$ as the closed ball with radius $R_3$ in the space of symmetric and positive semi-definite matrices -  this is a compact set. Using the  fact that the complement of 
$\{( \mathbf{K}^{(2,n)},  \mathbf{K}^{(3,n)} ) \in S_2 \times S_3 \}$ is contained in $\{ \mathbf{K}^{(2,n)} \notin S_2\} \cup \{ \mathbf{K}^{(3,n)} \notin S_3,  \mathbf{K}^{(2,n)} \in S_2\}$, we obtain
\begin{equation} \label{eq:cartesian_complement_decomposition}
    \PP\left( \left( \mathbf{K}^{(2,n)},  \mathbf{K}^{(3,n)}\right) \notin S_2 \times S_3\right) 
    \le 
    \PP\left(  \mathbf{K}^{(2,n)} \notin S_2 \right) + 
    \PP\left(  \mathbf{K}^{(3,n)} \notin S_3,  \mathbf{K}^{(2,n)} \in S_2 \right) 
\end{equation}
At this point the supremum bound for integrals gives
\begin{align*} \label{eq:}
    \PP\left(  \mathbf{K}^{(3,n)} \notin S_3 ,  \mathbf{K}^{(2,n)} \in S_2 \right) 
    & = \int_{S_2} \PP\left(  \mathbf{K}^{(3,n)} \notin S_3 \mid  \mathbf{K}^{(2,n)} = Q \right) \PP\left( \mathbf{K}^{(2,n)} \in  dQ \right)
    \\
    &\le 
    \sup_{Q\in S_2}
    \PP\left(  \mathbf{K}^{(3,n)} \notin S_3 \mid  \mathbf{K}^{(2,n)} = Q \right),
\end{align*}
where 
\begin{align*}
    \sup_{Q\in S_2}
    \PP\left(  \mathbf{K}^{(3,n)}  \notin S_3|  \mathbf{K}^{(2,n)} = Q \right) =
    \sup_{Q\in S_2}
    \PP\left( \left\| \sum_{k=1}^n \CG^{(2)}\left( \sqrt{Q}Z_{k} \right) \right\|_F > nR_3   \right)
\end{align*}
with $Z_{1},\ldots, Z_{n} \stackrel{\text{iid}}{\sim} \CN\left(0, \id_{ D_2}\right)$.
Hence, by Lemma  \ref{lemmatight}, there is $R_3$ such that 
\begin{equation*}
    \PP\left(  \mathbf{K}^{(3,n)} \notin S_3 ,  \mathbf{K}^{(2,n)} \in S_2 \right)
    \leq    \sup_{Q\in S_2}
    \PP\left( \left\| \sum_{k=1}^n \CG^{(2)}\left( \sqrt{Q}Z_{k} \right) \right\|_F > nR_3   \right)
\leq    e^{-n\eps}.
\end{equation*}
By the arbitrariness of $\eps$, the base case is proved with
$S_3=\{Q:\|Q\|_F\leq R_3\}$. 
Now suppose for some $\ell = 3,\ldots,L$ the sequence $\left(  \mathbf{K}^{(2,n)}, \ldots,  \mathbf{K}^{(\ell,n)}\right)$ is jointly exponentially tight.
Then for a fixed $\eps>0$ there exist compact sets $S_2,\ldots,S_\ell \subset\CS_D^+$ such that 
\begin{align*}
    \PP\left( \left( \mathbf{K}^{(2,n)},\ldots,  \mathbf{K}^{(\ell,n)}\right) \notin S_2 \times \cdots \times S_{\ell}\right) \le (\ell-1)
    e^{-n\eps}
\end{align*}
Define the set $S_{\ell+1}$ as the closed ball with radius $R_{\ell+1}$ in the symmetric and positive semi-definite matrices. 
Arguing as in  \eqref{eq:cartesian_complement_decomposition} one gets
\begin{align*}
    &\PP\left( ( \mathbf{K}^{(2,n)}, \ldots,  \mathbf{K}^{(\ell+1,n)}) 
        \notin S_2 \times \cdots \times S_{\ell+1} \right) \\
    &\le
    \PP\left( ( \mathbf{K}^{(2,n)}, \ldots,  \mathbf{K}^{(\ell,n)}) 
        \notin S_2 \times \cdots \times S_{\ell} \right) 
    +
    \PP\left(  \mathbf{K}^{(\ell+1,n)} \notin S_{\ell+1}
    ,  \mathbf{K}^{(\ell,n)} \in S_\ell \right).
\end{align*}
Following exactly the same reasoning made for the base case we can show that
\[
\PP\left( \mathbf{K}^{(\ell+1,n)} \notin S_{\ell+1} \mid  \mathbf{K}^{(\ell,n)} \in S_\ell \right) \le e^{-n\eps}
\]
for a suitably big  $R_{\ell+1}$. 
\endproof

\subsection{Proof of Theorem \ref{prop:gs_ldp0}}
We now have  all the tools to complete the  proof of the LDP. 

\begin{proposition}[LDP for the general structure] 
\label{prop:gs_ldp} $\! \!$  Let $\{\! (\mathbf{K}^{(2,n)},\! \ldots, \mathbf{K}^{(L+1,n)})\!\}_n$ be as in {\rm Definition \ref{def:general_structure}}.
       Let {\rm Assumptions \ref{A1}} and {\rm\ref{A2bis} hold}. Let $I_\ell$ be as in {\rm Proposition \ref{prop:ldp-continuity}}.
    Then $\left\{\left( \mathbf{K}^{(2,n)}, \ldots,  \mathbf{K}^{(L+1,n)}\right)\right\}_n$
    satisfies a LDP on $\psd{D_2}\times \dots \times \psd{D_{L+1}}$ with good rate function given by
    \begin{equation*}
        I_{2,\dots,L+1}\left(Q_2, \ldots, Q_{L+1}\right) 
        = 
        \alpha_1 I_{1}\left(Q_{2} \mid  \mathbf{K}^{(1)}\right)
        +
        \sum\limits_{\ell=2}^L
        \alpha_\ell I_{\ell}\left(Q_{\ell+1} \mid Q_{\ell}\right).
    \end{equation*}
    
\end{proposition}
\proof
The proof proceeds by induction, starting from  $( \mathbf{K}^{(2,n)},  \mathbf{K}^{(3,n)})$. 
Since $ \mathbf{K}^{(1)}$ is deterministic, the conditional LDP on $ \mathbf{K}^{(2,n)}$ established in Proposition \ref{prop:ldp-continuity} is actually a full LDP
with good rate function $ \alpha_1 I_{1}\left(Q_{2} \mid  \mathbf{K}^{(1)}\right)$.
By Proposition \ref{prop:ldp-continuity} the transition kernel  $\nu_{3,n}$
satisfy the conditional LDP continuity condition with rate function
$\alpha_2 I_{2}$. Hence, Proposition \ref{thm:joint_LDP_general}  implies that the measure 
$\mu_{3,n}$
defined as
\begin{equation*}
    \mu_{3,n}(B_2 \times B_3) 
    :=  
    \int_{B_2} \nu_{3,n}(B_3 \mid Q)
    \PP\left(   \mathbf{K}^{(2,n)} \in dQ \right)
    =
    \PP\left(   \mathbf{K}^{(2,n)} \in B_2,  \mathbf{K}^{(3,n)} \in B_3 \right),
\end{equation*}
follows a weak LDP with rate function 
\begin{equation*}
\mathcal{I}(Q_2,Q_3)= \alpha_1 I_{1}\left( Q_{2} \mid  \mathbf{K}^{(1)}\right)
    +\alpha_2 I_{2}\left( Q_{3} \mid  Q_{2 }\right) 
   =: I_{2,3}(Q_2,Q_3).
\end{equation*}
However, since we proved in Proposition \ref{prop:exptight-K} that $\left\{\left( \mathbf{K}^{(2,n)}, \mathbf{K}^{(3,n)} \right)\right\}_n$ is exponentially tight,  
we deduce that $I_{2,3}$ must be good and that the sequence must satisfy the full LDP by  Theorem 2.3(iii) in \cite{rezakhanlou2017}.
Now suppose that for a fixed $\ell = 3,\ldots, L$ the sequence 
$\left\{ \left( \mathbf{K}^{(2,n)},\ldots, \mathbf{K}^{(\ell,n)}\right)\right\}_n$ satisfies a LDP with good rate function
\begin{equation*}
        I_{2,\ldots,\ell}\left(Q_{2}, \ldots, Q_{\ell}\right) 
        = \alpha_1 I_{1}\left(Q_{2} \mid \mathbf{K}^{(1)}\right) + \sum_{k=2}^{\ell-1}\alpha_k
        I_{k}\left(Q_{k+1} \mid Q_k\right).
\end{equation*}
By Proposition \ref{prop:ldp-continuity} the transition kernel  $\nu_{\ell+1,n}:\CB(\psd{D_{\ell+1}})\times \psd{D_\ell} \to[0,1]$ 
satisfy the conditional LDP continuity condition with rate function
$\alpha_\ell I_{\ell}$. 
By Markov property one has
\begin{align*}
    \nu_{\ell+1,n}(B \mid Q_\ell)
    = 
    \PP\left(   \mathbf{K}^{(\ell+1,n)} \in B 
    \mid 
     \mathbf{K}^{(2,n)} = Q_2, \ldots,  \mathbf{K}^{(\ell,n)} = Q_\ell \right).
\end{align*}
Hence Proposition \ref{thm:joint_LDP_general} implies that
$\{\mathbf{K}^{(2,n)}, \ldots,  \mathbf{K}^{(\ell+1,n)}\}_n$
follows a weak LDP with rate function 
\begin{equation*}
    I_{2,\ldots,\ell+1}\left(Q_{2}, \ldots, Q_{\ell+1}\right) 
    = \alpha_1 I_{1}\left(Q_{2} \mid  \mathbf{K}^{(1)}\right) + \sum_{k=2}^{\ell}
    \alpha_k I_{k}\left(Q_{k+1} \mid Q_k\right)
    .
\end{equation*}
Similarly to what we have done for the base case, the exponential tightness of \newline 
$\left\{\left(  \mathbf{K}^{(2,n)}, \ldots,  \mathbf{K}^{(\ell+1,n)}\right)\right\}_n$ implies that $I_{2,\ldots,\ell+1}$ must be good and that the sequence satisfies the full LDP.
\endproof

In light of Lemma \ref{lemma:fromCNNtoGeneral}, Theorem \ref{prop:gs_ldp0}
follows from Proposition \ref{prop:gs_ldp}.

\section{Conclusions}

In this work we have developed a large deviation theory for convolutional neural networks in the infinite-channel regime. 
Our main result establishes a large deviation principle for the sequence of conditional covariance tensors associated with a broad class of CNN architectures, including multidimensional models with general receptive fields encoded through a patch--extractor formalism.

This extends the large deviation framework previously available for fully connected neural networks to the convolutional setting, where spatial structure and weight sharing introduce additional challenges. 
Despite these differences, the covariance tensor exhibits a tractable asymptotic behavior and admits an explicit variational characterization through the rate function.

As a byproduct, our analysis also recovers covariance concentration and Gaussian equivalence in a unified and streamlined way. 
We further establish large deviation principles under both the prior and posterior distributions, as well as for a suitably rescaled network output. 
In particular, the invariance of the rate function under the posterior highlights the ``lazy'' nature of the infinite--channel regime.

From a broader perspective, our results contribute to the probabilistic understanding of neural networks beyond Gaussian limits, providing a quantitative characterization of rare events and of atypical fluctuations via the large deviation principle. 

Several directions for future research remain open, including extending the analysis to non-Gaussian weight distributions under minimal assumptions, as well as investigating more complex asymptotic regimes in which additional parameters, such as network depth or sample size, also diverge. This line of research could further be broadened to encompass other interesting architectures, such as recurrent neural networks and transformers.


These results constitute a step toward a unified large deviation framework for modern neural network architectures.
\appendix

\section{Additional proofs}\label{AppendixA}

  \begin{proof}[Proof of lemma \ref{lemma:fromCNNtoGeneral}]
We reason component-wise and start from the first claim. 
First note that using \ref{A2_0} one has 
\begin{equation*}
\begin{split}
        \left| T^{(i,\ell)}_m \left(z_{:,\nu}\right) \right| 
    &\le
     A_T e^{B_T\|z_{:,\nu}\|_2^{r}} 
    \leq
    A_T e^{B_T \|z\|_2^{r}} .\\
    \end{split} 
\end{equation*}
 Recall that here we use the tensor notation 
but we see $z$ as a vector, so that both $\|z_{:,\nu}\|_2$ and 
$\|z\|_2$ need to be understood as vectors norm. In other words, $z_{:,\nu}$ should be considered as sub-vectors of $z$.
Hence for any $z \in \R^{N_\ell P}$, the component $G_{i,\mu,j,\nu}^{(\ell)}(z)$ can be bounded as follows:
\begin{align}
\label{eq:cnn_polybound}
    |G_{i,\mu,j,\nu}^{(\ell)}(z)| 
    &= \frac{1}{\lambda_\ell M_\ell}
    \left| 
    \sum_{m \in \CM_\ell} 
    T^{(i,\ell)}_m \left(z_{:,\mu} \right)
    T^{(j,\ell)}_m \left(z_{:,\nu} \right)
    \right| \nonumber\\
     & \leq \frac{A_T^2}{\lambda_\ell M_\ell}
\sum_{m \in \CM_\ell}
   e^{2B_T \|  z \|_2^{r}} \le \frac{ A^2_T}{\lambda_\ell} 
   e^{2B_T \|  z \|_2^{r}}, \nonumber
\end{align}
which gives 
\[
\|G^{(\ell)}(z)\|_F
= \left ( \sum_{i,j,\mu,\nu} |G_{i,\mu,j,\nu}^{(\ell)}(z)|^2 
\right)^{\frac{1}{2}} \leq 
 \frac{A^2_T N_\ell P}{\lambda_\ell} 
   e^{2B_T \|  z \|_2^{r}}.
\] 
This shows that  $G^{(\ell)}$ satisfies \ref{A2}.

In order to prove part (2), write
\begin{equation*}
    \begin{split}
    |G_{i,\mu,j,\nu}^{(\ell)}(z)& -G_{i,\mu,j,\nu}^{(\ell)}(z')| \\
& \leq  \frac{1}{\lambda_\ell M_\ell} \sum_{m \in \CM_\ell} \left| 
    T^{(i,\ell)}_m \left(z_{:,\mu}\right)    T^{(j,\ell)}_m \left(z_{:,\nu}\right)
    - T^{(i,\ell)}_m\left(z_{:,\mu}'\right)    T^{(j,\ell)}_m \left(z_{:,\nu}'\right)
    \right| \\
&
\leq  \frac{1}{\lambda_\ell M_\ell} \sum_{m \in \CM_\ell} \left| 
    T^{(i,\ell)}_m \left(z_{:,\mu}\right) - T^{(i,\ell)}_m \left(z_{:,\mu}'\right)  \right|  \left| T^{(j,\ell)}_m \left(z_{:,\nu}\right) \right| 
    \\
&+ \frac{1}{\lambda_\ell M_\ell} \sum_{m \in \CM_\ell} \left| 
 T^{(i,\ell)}_m \left(z_{:,\mu}'\right) \right|    \left|  T^{(j,\ell)}_m \left(z_{:,\nu}\right)-T^{(j,\ell)}_m \left(z_{:,\nu}'\right)\right| .
\\
    \end{split}
\end{equation*}
By  \ref{A2_0BIS}, for every $\varepsilon>0$, there exists a constant $C=C(\varepsilon)$ such that, for vectors  $z_1,z_2$,
\begin{equation}\label{eq:sigma_eps_bound}
|T^{(j,\ell)}_m(z_1)-T^{(j,\ell)}_m(z_2)|\le L_T\|z_1-z_2\|_2 + \varepsilon(\|z_1\|_2+\|z_2\|_2) + C.
\end{equation}
In particular, $T^{(j,\ell)}_m$ is linearly bounded; that is, there exists a constant $A>0$ such that
\begin{equation}\label{eq:sigma_linear_bound}
|T^{(j,\ell)}_m(z)| \le A(1+\|z\|_2), \qquad \forall\, z \in\mathbb{R}^{N_\ell}.
\end{equation}
Using the bound \eqref{eq:sigma_eps_bound}, we obtain 
\begin{align*}
\left|T^{(i,\ell)}_m\left(z_{:,\mu}\right) - T^{(i,\ell)}_m\left(z'_{:,\mu}\right)\right|
&\le 
 L_T\|z_{:,\mu}-z_{:,\mu}'\|_2
+
\varepsilon (\|z_{:,\mu}\|_2+\|z_{:,\mu}'\|_2) + C,
\\&\le 
L_T\|z-z'\|_2
+
\varepsilon (\|z\|_2+\|z'\|_2) + C,
\end{align*}
Using the linear bound \eqref{eq:sigma_linear_bound}, we deduce
\begin{align*}
&
\left|T^{(i,\ell)}_m \left(z_{:,\mu}\right)
      - T^{(i,\ell)}_m \left(z'_{:,\mu}\right)\right|
\,
\left|T^{(j,\ell)}_m \left(z_{:,\nu}\right)\right|
\\[4pt]
&\le 
A L_T \|z-z'\|_2(1+\|z\|_2)
+ 
\varepsilon  A (\|z\|_2 + \|z'\|_2 + \|z\|_2\|z'\|_2 + \|z\|_2^2)
+ 
A  C (1+\|z\|_2).
\end{align*}
Using $\|z\|_2\|z'\|_2 \leq \frac{1}{2}(\|z\|_2^2+\|z'\|_2^2)$  and $\|z\|_2\le 1+\|z\|_2^2$
one gets
\begin{equation*}
\begin{split}
& \left|T^{(i,\ell)}_m\left(z_{:,\mu}\right)
       - T^{(i,\ell)}_m\left(z'_{:,\mu}\right)\right|\left|
       T^{(j,\ell)}_m \left(z_{:,\nu}\right)\right| \\
&  \quad  \leq  A L_T \|z-z'\|_2(1+\|z\|_2)
+\varepsilon  A 
\Big [2+\frac{5}{2}\|z\|_2^2+\frac{3}{2}\|z'\|_2^2\Big]+   A   C (1+\|z\|_2). \\
&
 \quad  \leq  A  L_T \|z-z'\|_2\|z\|_2
+\varepsilon A  \Big [2+\frac{5}{2}\|z\|_2^2+\frac{3}{2}\|z'\|_2^2\Big] \\
& +  ( A   C+ A L_T)\|z\|_2+ A L_T\|z'\|_2+  A  C . 
\\
\end{split}
\end{equation*}
The analogous bound holds for the second term
$ |T^{(i,\ell)}_m\left(z_{:,\nu}\right)
      -T^{(i,\ell)}_m\left(z'_{:,\nu}\right)|
|T^{(j,\ell)}_m \left(z_{:,\mu}'\right)|
$ with $z$ in place of $z'$ and $z'$ in place of $z$.
Putting it all together, one gets 
\begin{equation*}
\begin{split}
& \left|G_{i,\mu,j,\nu}^{(\ell)}(z)-G_{i,\mu,j,\nu}^{(\ell)}(z')\right|
\le 
 C_1  \|z-z'\|_2(\|z\|_2+\|z'\|_2) \\
&\qquad  + 
\varepsilon  C_2 (\|z\|_2^2+\|z'\|_2^2)+ C_3(\|z\|_2+\|z'\|_2)+ C_4 \\
\end{split}
\end{equation*}
with $ C_1= A L_T/\lambda_\ell$,
$ C_2=4  A /\lambda_\ell$, $ C_3= A ( C +2 L_T)/\lambda_\ell$ and $ C_4=2( A C+\varepsilon  2 A)/\lambda_\ell$.
Using $ C_3 \|z\|_2\leq \eps\|z\|_2^2+ C_3^2/\eps$
\begin{equation*}
\begin{split}
     \left|G_{i,\mu,j,\nu}^{(\ell)}(z)-G_{i,\mu,j,\nu}^{(\ell)}(z')\right| &
  \leq 
 C_1  \|z-z'\|_2(\|z\|_2+\|z'\|_2)\\&
+ 
\varepsilon ( C_2+1) (\|z\|_2^2+\|z'\|_2^2)+ 2 C_3^2/\varepsilon+ C_4. \\
\end{split}
\end{equation*}
Since $C_2$ does not depend on $\varepsilon$, 
the claim follows. 
\end{proof}


\begin{proof}[Proof of Lemma \ref{lemma:bridge-lln}]
To prove that $Y_n \convp y_0$, it is sufficient to show that $Y_n$ converges in distribution to the constant $y_0$. This, in turn, is equivalent to showing that $\E[\phi(Y_n)] \to \phi(y_0)$ for any bounded and continuous function $\phi: \YY \to \R$ (shortened from now on as $\phi\in C_b(\YY)$).

Let $\phi \in C_b(\YY)$ be arbitrary and set 
\[
g_n(x): = \E[\phi(Y_n) \mid X_n = x]=\int \phi(y) \nu_{n}(dy|x).
\]
Let $\{n_k \}_k\to +\infty$ 
be any diverging  subsequence.
Since $X_n \convp x_0$, the subsequence $X_{n_k}$ also converges in probability to $x_0$. 
We can extract a further diverging subsequence $\{n_{k_j}\}_j$ such that $X_{n_{k_j}} \to x_0$ almost surely as $j \to \infty$. Call $\Omega_{0}$ the almost sure subset of realization such that $X_{n_{k_j}}(\omega) \to x_0$ for all $\omega \in \Omega_{0}$. Then, by the arbitrariness of the sequences in hypothesis (2), for every $\omega\in\Omega_{0}$ as $j\to\infty$ one has
\begin{equation*}
    g_{n_{k_j}}(X_{n_{k_j}}(\omega)) \to \phi(y_0)
\end{equation*}
meaning that $g_{n_{k_j}}(X_{n_{k_j}})$ converges almost surely to $\phi(y_0)$ as $j\to\infty$.
Moreover, the random variables $g_{n_{k_j}}(X_{n_{k_j}})$ are uniformly bounded since 
\begin{equation*}
|g_n(x)| = |\E[\phi(Y_n) \mid X_n = x]| \le \E[|\phi(Y_n)| \mid X_n = x] \le \|\phi\|_{\infty}.
\end{equation*}
Consequently, by the Dominated Convergence Theorem, when $j\to\infty$ we have
\begin{equation*}
\E\left[\phi\left(Y_{n_{k_j}}\right)\right] = \E\left[ \E\left[\phi\left(Y_{n_{k_j}}\right) \mid X_{n_{k_j}}\right] \right] =  \E\left[g_{n_{k_j}}\left(X_{n_{k_j}}\right)\right] \to \phi(y_0).
\end{equation*}
Since we have shown that from an arbitrary subsequence $\{n_k\}_k$ we can extract a further subsequence $\{n_{k_j}\}_j$ such that $\E[\phi(Y_{n_{k_j}})]$ converges to $\phi(y_0)$, we conclude that the entire sequence converges:
\[
\lim_{n\to\infty} \E[\phi(Y_n)] = \phi(y_0).
\]
As this holds for any $\phi \in C_b(\YY)$, we have established that $Y_n \convd y_0$. Since the limit is a constant, this is equivalent to convergence in probability, $Y_n \convp y_0$.
\end{proof}

\begin{proof}[Proof of Lemma \ref{lemma:triangle_lln}]
For notational simplicity, we assume $C_\ell(n)=n$. 
Using Lemma \ref{ineqNormaG} and the fact that $r<2$ one has 
$ \E\left[\|\CG^{(\ell)}\left(\sqrt{Q}Z \right)\|_F\right] \leq  A\E[ e^{B\|Q\|_2^{r/2}  \|Z\|_2^r} ] <+\infty$. Hence, 
    by the Strong Law of Large Numbers, we have
    \begin{equation*}
        \frac{1}{n}\sum_{c=1}^n \CG^{(\ell)}\left(\sqrt{Q}Z_c \right) \xrightarrow{\text{a.s.}} \E\left[\CG^{(\ell)}\left(\sqrt{Q}Z \right)\right].
    \end{equation*}
    We decompose the term of interest as follows:
    \begin{equation*}
        \frac{1}{n}\sum_{c=1}^n \CG^{(\ell)}\left(\sqrt{Q_n}Z_c\right) 
        =
        \frac{1}{n}\sum_{c=1}^n \CG^{(\ell)}\left(\sqrt{Q}Z_c\right)
        + 
        \Delta_n,
    \end{equation*}
    where
$ \Delta_n 
        :=
        \frac{1}{n}\sum_{c=1}^n \left[ \CG^{(\ell)}\left(\sqrt{Q_n}Z_c\right) 
        -
        \CG^{(\ell)}\left(\sqrt{Q}Z_c\right) \right]$.   We will now show that $\Delta_n \stackrel{\PP}{\rightarrow} 0$.  To this end  it is sufficient to prove that $\E\left[\|\Delta_n\|_2 \right] \to 0$. 
Remark \ref{continuitySqrt} and the
   continuity of $\CG^{(\ell)}$ implies $\CG^{(\ell)}\left(\sqrt{Q_n}z\right) \to \CG^{(\ell)}\left(\sqrt{Q}z\right)$ for every $z\in\R^{D_\ell}$, and hence
    \begin{equation*}
        \left\|
        \CG^{(\ell)}\left( \sqrt{Q_n}Z \right) 
        -
        \CG^{(\ell)}\left( \sqrt{Q}Z \right)
        \right\|_2
        \to
        0
        \quad \text{a.s.}.
    \end{equation*}
    Moreover, since 
   the sequence $\{Q_n\}_n$ converges and hence its norm is bounded,  
     by Lemma \ref{ineqNormaG} for some $M$ one has    
  \[
    \left\| \CG^{(\ell)}\left( \sqrt{Q_n} Z \right)  -         \CG^{(\ell)}\left( \sqrt{Q} Z\right) \right\|_2        
\le   2A  e^{ B M^{r/2} \left\|Z \right\|_2^{r} }.
\]
Now   observe that the right-hand side is an integrable random variable since $r<2$ and $Z$ is Gaussian.
    Therefore, by the Dominated Convergence Theorem, we can conclude that
   \[
        \E\left[ \left\| \CG^{(\ell)}\left( \sqrt{Q_n} Z \right) 
        - 
        \CG^{(\ell)}\left( \sqrt{Q} Z \right) \right\|_2 \right] \to 0.
        \]
    Finally, by the triangular inequality and the definition of $\{Z_k\}_k$ we have 
    \begin{align*}
        \E\left[ \left\|\Delta_n \right\|_2 \right] 
        &=
        \E\left[ \left\| \frac{1}{n}\sum_{c=1}^n \left( \CG^{(\ell)}\left( \sqrt{Q_n} Z_c \right) 
        - 
        \CG^{(\ell)}\left( \sqrt{Q} Z_c \right) \right) \right\|_2 \right]
        \\
        &\le
        \frac{1}{n}\sum_{c=1}^n \E\left[ \left\| \CG^{(\ell)}\left( \sqrt{Q_n} Z_c \right) 
        - 
        \CG^{(\ell)}\left( \sqrt{Q} Z_c \right) \right\|_2 \right]
        \\
        &=
        \E\left[ \left\| \CG^{(\ell)}\left( \sqrt{Q_n} Z \right) 
        - 
        \CG^{(\ell)}\left( \sqrt{Q} Z \right) \right\|_2 \right] \to 0.
    \end{align*}
    This shows that $\Delta_n \to 0$ in $L^1$, which implies $\Delta_n \stackrel{\PP}{\rightarrow} 0$. The result now follows.
\end{proof}

\bibliographystyle{abbrv}
\bibliography{biblioCNN}
\end{document}